\documentclass[12pt]{amsart}
\usepackage{amsmath}
\usepackage{amssymb}
\usepackage{amsfonts}
\usepackage{amsthm}
\usepackage{mathrsfs}
\usepackage[all]{xy}
\usepackage{color}
\usepackage{multicol,tikz-cd}
\usepackage{enumerate}
\usepackage{bbm}

\usepackage[vcentermath,enableskew]{youngtab} 
\usepackage{multirow}
\usepackage{tikz}
\usetikzlibrary{arrows}

\usepackage[colorlinks=true, pdfstartview=FitV, linkcolor=blue, citecolor=blue, urlcolor=blue, backref=true]{hyperref}
\newcommand{\longhookrightarrow}{\hookrightarrow}

\theoremstyle{plain}
\newtheorem*{TheoremA}{Theorem~A}
\newtheorem*{TheoremB}{Theorem~B}
\newtheorem*{TheoremC}{Theorem~C}
\newtheorem*{TheoremD}{Theorem~D}
\newtheorem{theorem}{Theorem}[section]
\newtheorem{lemma}[theorem]{Lemma}
\newtheorem{cor}[theorem]{Corollary}
\newtheorem{corollary}[theorem]{Corollary}
\newtheorem{proposition}[theorem]{Proposition}
\newtheorem{rem}[theorem]{Remark}

\theoremstyle{definition}

\newtheorem{remark}[theorem]{Remark}

\newtheorem{example}[theorem]{Example}

\newcommand{\C}{\mathbb{C}}

\renewcommand{\k}{\mathbbm{k}}
\newcommand{\lb}{{\lambda}}
\newcommand{\Vb}{{\mathbf{V}}} 
\renewcommand{\O}{\mathbb{O}}

\newcommand{\V}{\mathbb{V}}

\newcommand{\id}{{\rm id}}
\newcommand{\mcirc}{\mathfrak{m}^\circ}

\newcommand{\G}{\mathbf{G}}
\newcommand{\T}{\mathbf{T}}

\newcommand{\cN}{\mathcal{N}}
\newcommand{\Oc}{\mathcal{O}}
\newcommand{\bO}{\mathbf{O}}

\newcommand{\N}{\mathcal{N}}

\newcommand{\J}{\mathfrak{J}}

\DeclareMathOperator{\chr}{char}
\DeclareMathOperator{\End}{End}

\DeclareMathOperator{\GL}{GL}

\DeclareMathOperator{\gr}{\operatorname{gr}}
\DeclareMathOperator{\Spec}{\operatorname{Spec}}

\DeclareMathOperator{\Id}{Id}

\DeclareMathOperator{\Sp}{Sp}
\DeclareMathOperator{\Lie}{Lie}
\DeclareMathOperator{\Ind}{Ind}

\renewcommand{\l}{\mathfrak{l}}
\newcommand{\mm}{\mathfrak{m}}
\newcommand{\g}{\mathfrak{g}}
\newcommand{\gl}{\mathfrak{gl}}
\renewcommand{\sp}{\mathfrak{sp}}

\renewcommand{\t}{\mathfrak{t}}

\newcommand{\mf}{\mathfrak}
\newcommand{\gb}{\mathfrak{gl}_n}

\newcommand{\codim}{\operatorname{codim}}

\newcommand{\ov}{\overline}
\newcommand{\reg}{\operatorname{reg}}
\newcommand{\Gbreg}{{\G\operatorname{-reg}}}
\newcommand{\Greg}{{G\operatorname{-reg}}}

\def\into{\hookrightarrow}
\def\onto{\twoheadrightarrow}

\numberwithin{equation}{section}

\title[Jordan stratification of enhanced \& exotic modules]{Sheets, Jordan classes and induced orbits in the exotic and enhanced modules}

\author{Filippo Ambrosio}
\address{F.~Ambrosio: Institut f\"ur Mathematik und Informatik, FSU-Jena, Inselplatz 5, 07743 Jena, Germany}\email{filippo.ambrosio@uni-jena.de}

\author{Giovanna Carnovale}
\address{G.~Carnovale: Dipartimento di Matematica ``Tullio Levi-Civita'', University of Padova, via Trieste 63, 35121, Padova, Italy}\email{carnoval@math.unipd.it}

\author{Francesco Esposito}
\address{F.~Esposito: Dipartimento di Matematica ``Tullio Levi-Civita'', University of Padova, via Trieste 63, 35121, Padova, Italy}\email{esposito@math.unipd.it}

\author{Neil Saunders}
\address{N.~Saunders: Department of Mathematics, City St George's University of London,  Northampton Square
London EC1V 0HB,  UK}
\email{neil.saunders.3@citystgeorges.ac.uk}

\author{Lewis Topley}
\address{L.~Topley: Department of Mathematical Sciences, University of Bath, Claverton Down, Bath, BA2 7AY, UK}\email{lt803@bath.ac.uk}

\usepackage{fancyhdr}
\makeatletter
\let\runauthor\@author
\let\runtitle\@title
\makeatother
\begin{document}

\begin{abstract}
Kato developed an exotic Deligne--Langlands correspondence using a geometric model for the multiparameter affine Hecke algebra of type {\sf C}, based on his {\it exotic nilpotent cone}. Achar--Henderson and Springer showed that this exotic nilpotent is intimately related to another, apparently  simpler variety called the {\it enhanced nilpotent cone}. Each of these is defined as the Hilbert nullcone of a polar module, the exotic $\operatorname{Sp}_{2n}$-module and the enhanced $\operatorname{GL}_n$-module, respectively.

In this paper we conduct a detailed study of the geometry of these two modules, by introducing the Jordan stratification, simultaneously generalising classical results on the adjoint representation as well as the symmetric space associated to $(\mathfrak{gl}_{2n}, \mathfrak{sp}_{2n})$. One of the key tools we develop is the theory of induced orbits in the enhanced and exotic nilpotent cones, following the work of Lusztig--Spaltenstein. Our main application is a classification of sheets in these modules, inspired by a theorem of Borho.
\end{abstract}

\maketitle

\section{Introduction}

\subsection{The exotic and enhanced modules}
\label{ss:exoticenhancedintro}

Throughout the introduction we fix an algebraically closed field $\k$.

Fix $n > 0$ and $G := \Sp_{2n}(\k)$. In \cite{K1} Kato defined the {\it $\ell$-exotic module} (for $\k = \C$) to be the $G$-module $(\k^{2n})^{\oplus \ell} \oplus (\bigwedge^2 \k^{2n})$, where $\k^{2n}$ is the vector representation of $G$. He then defined the {\it $\ell$-exotic nilpotent cone} to be the Hilbert nullcone of the $\ell$-exotic module. Kato's purpose was to define an analogue of the Steinberg variety for the 2-exotic nilpotent cone, and use this to obtain a geometric realisation of the 3-parameter Hecke algebra of type ${\sf C}_n^{(1)}$, generalising the work of Lusztig  and Chriss--Ginzburg (see \cite{K1} and \cite[Ch. 7]{CG} for references). Kato shows in \cite[Theorem~1.2]{K1} that the $1$-exotic nilpotent cone has many favourable geometric properties, reminiscent of the nilpotent cone of the adjoint representation. Soon after, in \cite{K2} Kato showed that the adjoint nullcone $\cN(\Lie(G))$ can effectively be viewed as a flat deformation of the exotic 1-nilpotent cone in characteristic $p = 2$. In this article we denote the $1$-exotic module by $\V$, and refer to it as the {\it exotic module}. Its Hilbert nullcone is denoted $\cN(\V)$.

Now let $\G := \GL_n(\k)$. In \cite{ah} Achar--Henderson defined the {\it enhanced module} to be the $\G$-module $\Vb :=  \k^n\oplus\gl_n$, and they carried out a thorough study of its Hilbert nullcone, which they dubbed the {\it enhanced nilpotent cone}. They studied the geometry of orbit closures in the enhanced nilpotent cone, using semisimall resolutions of singularities. One of their main results describes the intersection cohomology of orbit closures via Shoji's bipartition Kostka polynomials.  They, and independently Springer, observed in \cite[\textsection 6]{ah} and  \cite[\textsection 4]{springer} that for $\chr(\k) \ne 2$ the enhanced nilpotent cone $\cN(\Vb)$ includes into the exotic nilpotent cone inducing a bijection of orbit spaces
\begin{eqnarray}
\label{eq:orbitbijection}
\cN(\Vb)/\G \overset{1\text{-}1}{\longrightarrow} \cN(\V)/G
\end{eqnarray}
Achar--Henderson demonstrated that $\G$-orbits in $\cN(\Vb)$ are parameterised by bipartitions, and they used the bijection \eqref{eq:orbitbijection} to describe the dominance ordering of $G$-orbits in $\cN(\V)$. A parametrisation of the orbits in $\cN(\V)$ in terms of marked partitions was obtained in \cite[Theorem 8.3]{K1}. One of the core themes of the present article is a comparison of the geometry of $\Vb$ and $\V$ generalising \eqref{eq:orbitbijection}.

In \cite{springer} Springer considered the invariant theory of $(\Vb, \G)$ and $(\V, G)$, describing the basic homogeneous generators of $\k[\Vb]^\G$ and $\k[\V]^G$, under the same assumption on $\chr(\k)$ as \cite{ah}. He independently proved \eqref{eq:orbitbijection}, and extended the result by showing that the bijection doubles dimensions of orbits.  
The exotic and enhanced nilpotent cones lead to geometry, combinatorics and representation theory very similar to that of the adjoint representation of $\Sp_{2n}$; in many senses, however, the theory is more elegant in the enhanced and exotic setting, bearing similarities to the geometry of the adjoint representation of $\GL_n$.

\subsection{Sheets, Jordan classes and induced orbits in the adjoint representation}
\label{ss:sheetsandclasses}
For the moment, suppose $G = \GL_n$ or $G$ is a connected, simply connected, simple algebraic group not of type {\sf A}, and that $\chr(\k)$ is good for $G$. The Lie algebra will be denoted $\g = \Lie(G)$. 
The theory we describe in the present section was pioneered by Borho \cite{bo-inv} over $\C$, and was later investigated over fields of positive characteristic by Spaltenstein \cite{Spalt82, Spalt84} and developed 
(under our hypotheses) by Premet--Stewart \cite{PS}.

The sheets of the adjoint representation are the irreducible components of the rank varieties
\begin{eqnarray}
\label{eq:rank-vars}
\g_{(k)} := \{x\in \g \mid \dim G x = k\} \ \text{ for } k \in \mathbb{N}.
\end{eqnarray}
They are locally closed subsets of $\g$ and there are finitely many of them. The classification of these sheets was the main result of \cite{bo-inv} and it goes via Jordan classes, also known as decomposition classes.

Let $x = x_s + x_n$ denote the Jordan decomposition of $x\in \g$.  We say that two elements $x, y\in \g$ are Jordan equivalent if there exists an element $g\in G$ such that $g  x_n = y_n$ and $g  \g^{y_s} = \g^{x_s}$, where $\g^{z}$ denotes the centraliser of $z\in\g$. The Jordan classes are equivalence classes of elements of $\g$ under this relation.

It is easy to see that Jordan classes are classified by $G$-orbits of pairs $(\l, \Oc)$, where $\l$ is a Levi subalgebra and $\Oc \subseteq \N(\l)$ is a nilpotent orbit. If $x = x_s + x_n$ then $(\l, \Oc) = (\g^{x_s}, G^{x_s}  x_n)$ is such a pair, and the Jordan class of $x$ is
\begin{eqnarray}
\label{eq:classicalJordanclasses}
G  (\mf{z}(\l)^{\operatorname{reg}} + \Oc) = G  (\mf{z}(\l)^{\operatorname{reg}} + x_n),
\end{eqnarray}
where $\mf{z}(\l)^{\operatorname{reg}}$ denotes the elements $z$ in the centre $\mf{z}(\l) \subseteq \l$ satisfying $\g^z = \l$ or, equivalently, those elements which attain the maximal $G$-orbit dimension over $\mf z(\l)$.

Since there are finitely many nilpotent orbits under our hypotheses \cite[\textsection 2.5\&\textsection2.6]{JaNO} and since the $G$-conjugacy classes of Levi subalgebras are classified by subdiagrams of the Dynkin diagram of $G$, it follows that there are finitely many Jordan classes.

The $G$-orbits in a Jordan class have constant dimension, so each sheet is a union of Jordan classes. Thus each sheet contains a dense class. To classify the sheets one must explain when a pair $(\l, \Oc)$ corresponds to a class which is dense in a sheet.  In turn, to describe this distinguished collection of classes one must use the theory of induced orbits, first introduced by Lusztig--Spaltenstein \cite{LS}.

Let $(\l, \Oc)$ be a pair as above and let $z\in \mf{z}(\l)$. If we pick a parabolic subgroup $P \subseteq G$ with infinitesimal Levi decomposition $\Lie(P) = \l \ltimes \mf{u}_P$, then we can form the associated fibre bundle over the partial flag variety, with collapsing map to $\g$ induced by the $G$-action
\begin{equation}
\label{eq:classical-induction}
\begin{tikzcd}
G \times^P (z + \Oc + \mf{u}_P) \arrow[d, swap," "] \arrow[rr, " "]&& \g \\
G/P && 
\end{tikzcd} 
\end{equation}
The horizontal arrow is proper and finite (hence closed), and the image contains a unique dense orbit called the induced orbit, denoted $\Ind_{\l, z}^\g(\Oc)$. As the notation suggests it does not depend on the choice of $P$. If a nilpotent orbit can not be induced from an orbit in a proper Levi subalgebra then it is called rigid. It turns out that the Jordan class associated to $(\l, \Oc)$ is dense in a sheet if and only if $\Oc$ is rigid in $\l$, see \cite[Satz~4.3]{bo-inv} and \cite[Theorem~2.8]{PS}. Thus the sheets of $\g$ are classified by pairs $(\l, \Oc)$ where $\Oc \subseteq \l$ is a rigid nilpotent orbit.

\subsection{Main results}
\label{ss:mainresults}

Resume the notation $G, \G, \V, \Vb$ from Section~\ref{ss:exoticenhancedintro}. The purpose of this paper is to generalise the theory described in Section~\ref{ss:sheetsandclasses} to the enhanced and exotic modules. For the rest of the paper we let $\k$ be an algebraically closed field of characteristic $\ge 0$. Whenever we work with $(\V, G)$ we assume that $\chr(\k) \ne 2$.

The orbit spaces $\Vb/\G$ and $\V/G$ are each equipped with a poset structure (the ``dominance order'') determined by inclusion of orbit closures. Our first theorem concerns quotients of $\Vb$ and $\V$, see Theorem~\ref{prop:double-dim} and Theorem~\ref{T:quotients}. The first part generalises \eqref{eq:orbitbijection}, whilst the second part generalises Kostant's theorem on the adjoint quotient map (see Propositions~7.13 and 8.5 of \cite{JaNO}).
\begin{TheoremA}
\begin{enumerate}
\setlength{\itemsep}{4pt}
\item The inclusion $\Vb \subseteq \V$ induces a poset isomorphism which doubles the dimensions of orbits
$$\Vb/ \G \to \V/G.$$
\item The quotient maps 
\begin{eqnarray*}
\Vb & \longrightarrow & \Vb/\!/\G, \\ \V & \longrightarrow  &\V/\!/ G,
\end{eqnarray*}
 are faithfully flat. The fibres are irreducible, normal complete intersections, each consisting of finitely many orbits.
\end{enumerate}
\end{TheoremA}

In \cite{GV} Gatti--Viniberghi introduced a suitable generalisation of the Jordan decomposition, which makes sense  in any module over an algebraic group. Its existence in characteristic $0$ is proved in \cite[Appendix]{Kac}, and the proof can be adapted to positive characteristic using \cite{BR}.  We call this a Jordan--Kac--Vinberg decomposition. The foundational works \cite{K1, springer} identify a suitable notion of semisimple and nilpotent elements in $\Vb$ and $\V$, which we recall in Section~\ref{ss:nilpandss}. 
 In Section~\ref{ss:JKVdecomposition} we explain that this leads to a Jordan--Kac--Vinberg decomposition for the enhanced and exotic modules. Although this is not the unique such decomposition (Remark~\ref{r:JKVnonunique}) it is canonical being the only ones compatible with the natural  $\G$-module inclusion $\gl_n\into \Vb$ and the natural $G$-module inclusion $\V\into \k^{2n}\oplus \gl_{2n}$. 

In Section~\ref{ss:EnhancedExoticClasses} we introduce the enhanced and exotic Jordan classes. 
If $(v,x) \in \Vb \cong \k^n \oplus \gl_n$, and $x = x_s + x_n$ is the Jordan decomposition, then the {\it enhanced Jordan class} of $(v,x)$ is
\begin{equation}\label{eq:intro-enhanced-jordan-class}
{\mathfrak J}(v,x) := \G(v,(\Vb^{\G^{x_s}})^{\operatorname{reg}}+x_n),
\end{equation}
where $(\Vb^{\G^{x_s}})^{\reg}$ denotes the set of elements in the fixed-point set $\Vb^{\G^{x_s}}$ which have $\G$-orbit of maximal dimension. Similarly if $(v,x) \in \V \cong \k^{2n} \oplus \bigwedge^2 \k^{2n}$ then we can decompose $x = x_s + x_n$ using the  $\bigwedge^2 \k^{2n} \subseteq \gl_{2n}$, and define the {\it exotic Jordan class} of $(v,x)$ to be
\begin{equation}\label{eq:intro-exotic-jordan-class}
J(v,x) :=G(v,(\V^{G^{x_s}})^{\operatorname{reg}}+x_n).
\end{equation}
Again, the subset $(\V^{G^{x_s}})^{\operatorname{reg}}$ of the fixed-point subset $\V^{G^{x_s}}$ of $\V$ is the collection of elements with $G$-orbit of maximal dimension.
We invite the reader to compare \eqref{eq:classicalJordanclasses} with \eqref{eq:intro-enhanced-jordan-class} and \eqref{eq:intro-exotic-jordan-class}.

In Proposition~\ref{prop:basic-properties} we record the combinatorial classification of Jordan classes, and in Section~\ref{sec:geometryofclasses} we conduct a local study of the enhanced and exotic Jordan classes. By our Theorem~\ref{thm:final-poset-bijection}(1)\&(2) the closure of a Jordan class is a union of Jordan classes, and so one can equip the set of Jordan classes with the dominance order. We summarise our results on the geometry and combinatorics of classes as follows.
\begin{TheoremB}
\begin{enumerate}
\setlength{\itemsep}{4pt}
\item The  enhanced Jordan classes  are classified by $\G$-orbits of pairs $(L, \bO)$ where $L = \G^{x_s} \subseteq \G$ is a stabiliser of a semisimple element of $\Vb$ and $\bO \subseteq \cN(\Vb_L)$ is an orbit in the enhanced $L$-module. Identical remarks hold for the exotic Jordan classes.
\item Classes in both $\V$ and $\Vb$ are classified by pairs $(\lambda, \Oc_{\mu})$ where $\lambda = (\lambda_1,\ldots,\lambda_m)$  is a partition of $n$ in which $i\in\mathbb N$ has multiplicity $d_i(\lb)$, 
and $\Oc_\mu$ is an orbit  for the natural action of the direct product of the symmetric groups  $\mathbb S_{d_i(\lb)}$ on the set of sequences $(\mu^{(1)},\dots,\mu^{(m)})$  with  $\mu^{(i)}$ a bipartition of $\lambda_i$. 
In particular, there are finitely many classes.
\item The Jordan classes in $\Vb$ and $\V$ are locally closed, smooth and their dimensions can be computed using \eqref{eq:dimension-enhanced} and \eqref{eq:dimension-exotic}.
\item The inclusion $\Vb \subseteq \V$ induces a poset isomorphism from the Jordan classes in $\Vb$ to the Jordan classes in $\V$.
\end{enumerate}
\end{TheoremB}

In Section~\ref{sec:inductionandclosures-enhanced} we introduce the notion of induced orbits in the enhanced setting, following the formalism indicated in \eqref{eq:classical-induction} (Cf. \eqref{eq:associated-bundle}). We go on to show that induction of orbits in $\Vb$ satisfies all of the desirable properties of classical Lusztig--Spaltenstein induction.

In contrast to the enhanced case, we define induction of orbits in the exotic case using Jordan classes. A result of Borho \cite{bo-inv} states that an induced  nilpotent orbit in the adjoint representation $\g$ can be characterised as the unique orbit which is dense in the intersection of the nilpotent cone and Jordan class parameterised by the induction data (see also \cite[\textsection 2.5]{PS}). 
We use this characterisation to define induction of exotic and enhanced nilpotent orbits, only using the formalism \eqref{eq:classical-induction} when inducing enhanced orbits of elements consisting of a semisimple endomorphism and a vector. The latter is enough to describe the closure of an enhanced class, and the exotic case comes for free, by Theorem~B(4).

Here is a summary of our results on induced orbits. The precise statements are contained in Sections~\ref{sec:inductionandclosures-enhanced} and \ref{sec:inductionandclosures-exotic}.

\begin{TheoremC}
Let $x_s \in \Vb$ or $x_s \in \V$ be semisimple elements. Let $\Vb_{\G^{x_s}}$ and $\V_{G^{x_s}}$ denote the enhanced and induced modules for the stabilisers. In Sections~\ref{sec:inductionandclosures-enhanced} \& \ref{sec:inductionandclosures-exotic} we define induction of orbits
\begin{eqnarray*}
\Ind_{\Vb_{\G^{x_s}}}^{\Vb} & : & \cN(\Vb_{\G^{x_s}})/\G^{x_s} \longrightarrow \cN(\Vb)/\G,\\
 \Ind_{\V_{G^{x_s}}}^{\V} & : & \cN(\V_{G^{x_s}})/G^{x_s} \longrightarrow \cN(\V)/G.
\end{eqnarray*}
\begin{enumerate}
\setlength{\itemsep}{4pt}
\item Induction is transitive.
\item Induction preserves codimension of orbits inside the ambient module.
\item The rigid orbits of $\mathcal N(\Vb)$ are those which cannot be induced from a proper submodule of the form $\mathcal N(\Vb_{G^{x_s}})$, and similar for $\V$. If $(v,x)$ lies in $\mathcal N(\Vb)$ or $\mathcal N(\V)$ then the orbit of $(v,x)$ is rigid if and only if $x = 0$.
\item $\G^{x_s}$ is a Levi subgroup of $\G$ and $\Ind_{\Vb_{\G^{x_s}}}^{\Vb}$ is independent of any choice of parabolic subgroup containing $\G^{x_s}$ as a Levi factor.
\item Both the closure and the regular closure of any Jordan class can be described in terms of induced orbits, see Proposition~\ref{prop:closure-general}.
\end{enumerate}
\end{TheoremC}

As we mentioned in Section~\ref{ss:sheetsandclasses}, one of the notable applications of the theory of Jordan classes in the adjoint representation is the classification of sheets. For $k \in \mathbb{N}$ the rank varieties $\Vb_{(k)}$ and $\V_{(k)}$ are defined identically to \eqref{eq:rank-vars}. The sheets of enhanced and exotic modules are defined to be the irreducible components of their respective rank varieties. In Remark~\ref{rem:classification-sheets} we combine our results on classes and induction to classify these sheets, following the method of Borho \cite{bo-inv}.

Let $\J(L, \bO) \subseteq \Vb$ denote the enhanced Jordan class associated to the $\G$-orbit of the pair $(L, \bO)$ in Theorem~B(1) and let $\Vb_L$ be the enhanced $L$-module. Similarly write $J(G^{x_s}, \O) \subseteq \V$ for the exotic Jordan class associated to the $G$-orbit of the  pair  $(G^{x_s}, \O)$  and $\V_{G^x_s}$  for the exotic $G^{x_s}$-module. 

\begin{TheoremD}
The sheets of $\Vb$ are precisely the sets
\begin{eqnarray}
\label{e:sheets-enhanced}
\{ \overline{\J}^{\reg} \mid \J = \J(L, \bO), \ L \subseteq \G \text{ Levi}, \ \bO \subseteq \cN(\Vb_L) \text{ rigid}\}.
\end{eqnarray}
The sheets of $\V$ are precisely the sets
\begin{eqnarray}
\label{e:sheets-exotic}
\{ \overline{J}^{\reg} \mid J = J(G^{x_s}, \O), \ x_s \in \V \text{ semisimple}, \O \subseteq \cN(\V_{G^{x_s}}) \text{ rigid}\}.
\end{eqnarray}
The inclusion $\Vb \to \V$ induces a bijection between the sets \eqref{e:sheets-enhanced} and \eqref{e:sheets-exotic}. Under the parameterisation of Theorem~B(3) both sets are in bijection with pairs $(\lambda, \mu)$ such that for each $i =1,...,m$ we either have $\mu_i = (1^{\lambda_i} ; \emptyset)$ or $\mu_i = (\emptyset ; 1^{\lambda_i})$.
\end{TheoremD}

\subsection{Survey of some related works}

In \cite{mirabolic} Travkin developed a Robinson--Schensted--Knuth correspondence with a view to studying mirabolic character sheaves on $\GL_n \times \C^n$. In the first step he gave a conceptual explanation for the bijection \eqref{eq:orbitbijection}, that he obtained independently. The finitude of orbits in $\cN(\Vb)$ was observed earlier by Gan--Ginzburg \cite[Corollary~2.2]{GG}, and it is a consequence of a result by Berstein in \cite[\textsection 4.1]{bernstein}. Stabilisers of elements in the enhanced and exotic module were studied by Springer in \cite{springer};  in the enhanced nilpotent cone by Sun in \cite{sun}, and in the exotic nilpotent cone by Kato in \cite{K1}. A different parametrisation of the nilpotent orbits in the enhanced nilpotent cone, in terms of quiver representations, was obtained by Bellamy and Boos in \cite[\textsection 4]{BB}, with the goal of addressing the enhanced cyclic nilpotent cone and  admissible $\mathcal D$-modules on the space of representations of related framed cyclic quivers.  
In \cite{ah} Achar--Henderson suggested that the enhanced Springer fibres would be paved by affine spaces. Such a property provides a fundamental method for calculating intersection cohomology of fibres (see \cite[\textsection 11]{JaNO}). This expectation was later confirmed by Mautner \cite{Ma}. In a subsequent work Achar--Henderson--Jones conjectured that the enhanced orbit closures are normal, and proved it in numerous cases \cite{ahj}.

The theory of sheets has been generalised to the setting of symmetric pairs by Bulois \cite{Bul}. Our results generalise those of Bulois concerning the symmetric pair $(\mathfrak{gl}_{2n}, \mathfrak{sp}_{2n})$. 
Later Bulois--Hivert augmented this study \cite{BH} by defining a version of induced nilpotent orbits in symmetric spaces, using a technique they call slice induction.

Finally,  we point out the recent preprints \cite{Antor_g2, Antor_f4}, where Antor investigates exotic nilcones in type $F_4$ and $G_2$, associated Springer correspondences and their relations to affine Hecke algebras with unequal parameters. It would be interesting to verify if the same program carried out in the present paper could be extended to other polar or exotic representations.

\tableofcontents

\section{Preliminaries}

\subsection{Notation and conventions}

For  an algebraic group $H$ acting on a variety $X$, and $Y\subset X$, we denote by $Y^{H\operatorname{-reg}}$ the regular locus of $Y$,  i.e.,  the open dense subset consisting of points whose $H$-orbit has maximal dimension.
When the group acting is clear from the context, we simply write $Y^{\operatorname{reg}}$.  Note that the notion of $H$-regular elements of a subset of $Y$ does not require the subset $Y$ to be $H$-stable. 
Also, we denote by $N_H(Y)$ the stabiliser of $Y$ in $H$ and by $H^Y$ the pointwise stabiliser of $Y$ in $H$. If $Y=\{y\}$ is a single point we write $H^y$.

For an affine variety $X$ acted upon by a reductive group $H$, we set $X/\!/H:={\rm Spec}(\k[X]^H)$.
There is a natural projection $\pi_{X,H}\colon X\to X/\!/H$.  A subset $Y\subset X$ is called $\pi_{X,H}$-{\em saturated} if $Y=\pi_{X,H}^{-1}\pi_{X,H}(Y)$.
A saturated set is automatically $H$-stable but the converse does not hold in general.

Let $H$ be an algebraic group acting on a variety $X$, and let $M\leq H$ be a closed subgroup and $Y\subset X$ an $M$-stable subvariety.  
Then $M$ acts freely on $H\times Y$ via $m (h,y)=(hm^{-1}, m.y)$ for $m\in M$, $h\in H$ and $y\in Y$.
The quotient for this action, denoted by $H\times^M Y$, is the fibre bundle associated with the principal fibre bundle $H\to H/M$ and the $M$-action on $Y$.
Its elements are denoted by $h*y$, for $h\in H$ and $y\in Y$.
The group $H$ acts on  $H\times^M Y$ via left multiplication. 
The assignment $Z\mapsto H\times^MZ$ determines a bijection between $M$-stable affine  open subsets of $Y$ and $H$-stable open subsets of $H\times^M Y$, \cite{BR}.
We have $\k[H\times^MY]^H=\k[Y]^M$ and so $(H\times^M Y)/\!/H\simeq Y/\!/M$.  \
We remark that if $H$ is connected reductive and $Y$ is smooth, since the  projection  $H\to H/M$ is a  smooth morphism and the fibres of the bundle $H\times^ MY\to H/M$ are isomorphic to $Y$, the bundle $H\times^ MY$ is smooth. 

For $n$ a positive integer, a partition $\mu =(\mu_1, \dots, \mu_l)$ of $n$ is a sequence of nonincreasing positive integers $\mu_1 \geq \dots \geq \mu_l$ whose sum is $n$.
We denote by $\ell(\mu) = l$ the length of the partition. A bipartition on $n$ is a pair of partitions $(\mu ; \nu)$ such that $|\mu| + |\nu| = n$. 
We denote  by $(\mu\cup\mu;\nu\cup\nu)$ the bipartition of $2n$ obtained by doubling all terms in $\mu$ and $\nu$, and by $\mathcal{Q}_{n}$ the set of bipartitions of $n$. 

\subsection{The enhanced and exotic modules of Achar--Henderson and Kato}\label{ss:inclusions}

Fix once and for all $n > 0$ and let $\G = \GL_n$ and $\gb$ its Lie algebra.
 Throughout this article we let $e_1,...,e_n$ be the standard basis for the natural representation $V = \k^n$, and $e_1^*,...,e_n^*$ the dual basis of $V^*$.  We consider the cotangent vector space $T^*V=V\oplus V^*$, with its symplectic structure, and we identify $T^*V\simeq \k^{2n}$ by ordering the basis as $\{e_1,\,\ldots,\,e_n,e_1^*,\,\ldots,\,e^*_n\}$.
 
 Let $G:= \Sp_{2n}$ be the corresponding symplectic group, with Lie algebra $\g=\sp_{2n}$. 
The map $\G \hookrightarrow \End(T^*V)$ identifies $\GL_n$ as a subgroup of $\Sp_{2n}$ via $A\mapsto {\rm diag}(A, \,^t\!A^{-1})$. 
Denote by $\T \subset \G$  (resp. by $T \subset G$) the maximal torus of diagonal matrices.
Under the embedding $\G \subset G$ our chosen maximal tori $\T$ and $T$ are identified.

In \cite{ah} Achar--Henderson introduced the {\it enhanced $\G$-module}, defined by
\[\Vb := \Vb_n = V \oplus (V\otimes V^*) \simeq \k^n \oplus\gb.\]
Kato's {\it exotic module} \cite{K1} for $G$ is defined as
\begin{eqnarray}
\V :=  T^*V \oplus \textstyle{\bigwedge^2} T^* V\simeq \k^{2n}\oplus \textstyle{\bigwedge^2} \k^{2n}.
\end{eqnarray}

There is a one-dimensional (trivial) $G$-submodule of $\bigwedge^ 2 T^*V$, namely $\k\omega = \k \sum_{i=1}^n e_i^* \wedge e_i$. 
The quotient $\V/\k\omega$ was extensively studied by Kato \cite{K1}. When $\k$ is the field complex numbers, $\V$ is a well-known example of polar representation (see \cite{DK}).

We remind the reader that the characteristic of $\k$ is not 2 when we work with $(G,\V)$. As a $G$-module, the adjoint representation  $\gl_{2n}$ decomposes as $\gl_{2n} = \sp_{2n}\oplus\bigwedge^2 \k^{2n}$. Following \cite{springer} we identify the subspace  $\bigwedge^2 \k^{2n}$  with the set of 
endomorphisms of $\k^{2n}=T^*V$ that are self-adjoint with respect to the symplectic structure. Concretely, we view $\bigwedge^2 \k^{2n}$ as the space of  of matrices in $\gl_{2n}$ of the form 
\begin{align}\label{eq:self-adjoint}
\left(\begin{smallmatrix}A&B\\C&^t\!A\end{smallmatrix}\right),&&A,\,B,\,C\in \gl_n,  &&^t\!B=-B,\ ^t\!C=-C.\end{align}  

The identification of $T^* V$ with the natural module for $\GL_{2n}$  induces thus an inclusion $\V \hookrightarrow \k^{2n}\oplus \gl_{2n}=\Vb_{2n}$, which is $G$-equivariant. For this reason we may (and shall) view elements in $\V$ as pairs $(v,x)$ where $v\in \k^{2n}$ and $x\in \gl_{2n}$.

Similarly, we have a $\G$-equivariant embedding  of $\Vb$ into $\V$ through the natural  inclusions $V\into T^*V$ and $\gl_n\into \gl_{2n}$ given by  $A\mapsto \left(\begin{smallmatrix}A&0\\0&^t\!A\end{smallmatrix}\right)$.

\medskip

In summary, we have a chain of natural inclusions
\begin{eqnarray}
\label{eq:towerofinclusions}
\Vb_n \overset{\GL_n}{\longhookrightarrow} \V \overset{\Sp_{2n}}{\longhookrightarrow} \Vb_{2n} \longhookrightarrow \cdots
\end{eqnarray}
stemming from the following $G$ and $\GL_{2n}$-equivariant morphisms, respectively
\medskip 
\begin{align}\label{eq:inclusions}
\begin{array}{rlccrl}
\varphi \colon & G\times^{\G}\Vb \longrightarrow \V & \hspace{70pt} &\psi  \colon& \GL_{2n}\times^{G}\V &\longrightarrow \Vb_{2n} \\
 & g *(v,x) \longmapsto g(v,x) & \hspace{70pt} & & g *(v,x)&\longmapsto g(v,x).
\end{array}
\end{align}

When there is no ambiguity we will use the inclusions  \eqref{eq:towerofinclusions}, whilst we will use  the morphisms  \eqref{eq:inclusions} when there is need to underline differences. In this case, we will often write $\varphi(v,x)$, respectively $\psi(v,x)$  instead of $\varphi(1*(v,x))$, respectively $\psi(1*(v,x))$ to lighten notation.

 \medskip

For any partition $\lb=(\lb_1,\,\ldots,\,\lb_l)$ of $n$ we  denote by  $\G_\lb$ or $\GL_\lb$ the product $\GL_{\lb_1}\times\cdots\times\GL_{\lb_l}$. We also denote by $G_\lb$ or $\Sp_{2\lb}$ the product $\Sp_{2\lb_1}\times\cdots\times\Sp_{2\lb_l}$. Similarly we denote by $\Vb_\lb$ the enhanced module associated to $\G_\lb$, and by $\V_\lb$ the exotic module associated to $G_\lb$.

\subsection{Nilpotent and semisimple elements in the exotic and enhanced modules}
\label{ss:nilpandss}

In order to introduce a Jordan decomposition in the enhanced and exotic module, we recall that an element $x$ in the Lie algebra $\mathfrak h$ of a linear algebraic group $H$ defined over  $\Bbbk$ is semisimple if $x \in \Lie(S)$ for some torus $S \subset H$ and it is nilpotent  if $x \in \Lie (U)$ for some closed unipotent subgroup $U \subset G$,  see \cite[Remark 4.9, Propositions 11.8 \& 14.26]{borel}.
Any element in $\mathfrak h$ admits a unique Jordan-Chevalley decomposition as $x = x_s + x_n$, with $x_s$ semisimple, $x_n$ nilpotent and $[x_s, x_n] =0$,  we refer the reader to \cite[\S 4.4 Theorem 2]{borel}.

\medskip

Let $H$ be an algebraic group and $W$ a finite dimensional $H$-module. Following \cite{GV} we say that an element $w\in W$ is {\it semisimple} if the orbit $H\cdot w$ is closed, and we say that $w$ is nilpotent if $0\in \overline{H\cdot w}$. In case $W$ is the adjoint representation of $H$ these definitions coincide with the ones given in the previous paragraph. The nilpotent cone of $W$ is the variety of nilpotent elements, denoted $\cN(W)$. It is the vanishing set of all non-constant $H$-invariant polynomials.

\medskip

The nilpotent cones of $\Vb$ and $\V$ have been studied extensively in \cite{ah,K1,springer,mirabolic}. We have
\[\mathcal N(\Vb)=\mathcal \k^n \times \mathcal N(\gl_n),\quad \mathcal N(\V)=\mathcal \k^{2n}\times \mathcal N(\bigwedge\!\!{}^2\k^{2n})\]
where the elements in $\bigwedge^2\k^{2n}$ are nilpotent if and only if their associated matrix $\left(\begin{smallmatrix}A&B\\C&^t\!A\end{smallmatrix}\right)$ in $\gl_{2n}$ is nilpotent, \cite[\S 1.2\& Lemma 1.5]{springer}. 

Both subvarieties are stratified by nilpotent orbits and the corresponding orbit sets ${\mathcal N}(\Vb)/\G$ and ${\mathcal N}(\V)/G$ are partially ordered by inclusion of closures. 

\medskip

We introduce some useful notions from \cite{ah,mirabolic,springer}. The $\G$-orbits on $\mathcal{N}(\Vb)$  are finitely-many, and in bijection with bipartitions of $n$. We illustrate the bijection for later purposes. Let $(\mu;\nu)$ be a bipartition of $n$ and $(v,x)$ a representative for the associated orbit $\bO_{(\mu;\nu)} \subseteq \mathcal{N}(\Vb)$. Then we may define a normal basis for the nilpotent element $x$ on $\k^n$ as follows:
$$
\{v_{ij} \; | \; 1 \leq i \leq \ell(\mu + \nu) \; \text{and} \; 1 \leq j \leq \mu_{i} + \nu_{i} \} \quad \text{is a basis for $\k^n$}  
$$
such that
$$
	x v_{ij} = 
	\begin{cases}
		v_{i, j-1} \quad & \text{if} \; j>1, \\
		0 \quad & \text{if} \; j=0,
	\end{cases}
	\qquad \text{and} \qquad v = \sum_{i=1}^{\ell(\mu)} v_{i, \mu_{i}}
$$
That is, $x$ is in Jordan normal form with respect to the basis $\{v_{ij}\}$ and $v$ is written as a sum of the basis elements corresponding to the partition $\mu$. We say that the element $(v,x)$ has type $(\mu;\nu)$.

\begin{remark}
\label{rem:normal-form}
For computational purposes it is sometimes more convenient to  replace the Achar-Henderson representative $(v,x)$ by  $g(v,x)=(gv,gxg^{-1}) = (gv, x)$, for $g=\sum_{j=0}^{n}x^j$. i.e., to use the representative  $(\sum_{i=1}^{\ell(\mu)} \sum_{j=1}^{\mu_i}v_{i, j},x)$.
\end{remark}

\medskip

The posets ${\mathcal N}(\Vb)/\G$ and  ${\mathcal N}(\V)/G$ have a remarkably strong relationship, observed by Achar--Henderson and Springer.

\begin{lemma} \cite[Proposition 2.8\& Section 6]{ah} and \cite[Corollary~4.13]{springer} \label{lem:poset-nil}
\begin{enumerate}
\setlength{\itemsep}{4pt}
\item The restriction of the morphism $\varphi$ to $G\times^{\G}{\mathcal N}(\Vb)$ induces a  poset isomorphism  between ${\mathcal N}(\Vb)/\G$ and ${\mathcal N}(\V)/G$ whose inverse  is induced by taking the intersection with $\Vb$ via \eqref{eq:inclusions}. 
In addition, we have $$\dim G(v,x)=2\dim \G(v,x)$$ for any $(v,x)\in {\mathcal N}(\V)$. 
\item The  restriction of the morphism $\psi$ to $\GL_{2n}\times^{G}{\mathcal N}(\V)$ induces a poset inclusion from ${\mathcal N}(\Vb)/\G$ to ${\mathcal N}(\Vb_{2n})/\GL_{2n}$, mapping the $\G$-orbit labelled by  $(\mu;\nu)$ to the $\GL_{2n}$-orbit in $\mathcal N(\Vb_{2n})$ labelled by $(\mu\cup\mu;\nu\cup\nu)$. In addition,  
$$\dim \GL_{2n}\psi(\varphi(v,x))=2\dim G(v,x)=4\dim \G(v,x)$$ for any $(v,x)\in\mathcal N(\Vb)$. 
\end{enumerate} \end{lemma}

\begin{rem} \label{rem_exoticnotccs}
{\rm When $\Bbbk = \mathbb{C}$, in analogy to the  adjoint case, one may wonder if the closure of an exotic nilpotent orbit, with the natural $\mathbb{C}^\times$-action, is a conical symplectic singularity as defined by Beauville \cite{Beau}. This is not  the case in general, at least under the natural assumption that the symplectic structure is homogeneous with respect to the  contracting action. As a counterexample, consider the
exotic nilpotent orbit $\mathbb{O} \subset \mathbb{V}$ indexed by the bipartition $(\emptyset ; 2)$  for $G = \Sp_4$.
Then $(v,x) \in \overline{\mathbb{O}}$ if and only if $v = 0$ and $ x = \left( \begin{smallmatrix}
 a & b & 0& e \\
 c &-a  &-e& 0\\
 0 &f& a &c \\
 -f& 0 &b &-a
\end{smallmatrix} \right)$,
where $a,b,c,e,f \in \mathbb{C}$ are subject to $a^2+bc-ef=0$. Therefore  $\overline{\mathbb{O}}$ is isomorphic to a four-dimensional quadric hypersurface.
Assume now for a contradiction that the latter is a conical symplectic singularity with homogeneous contracting action.
By  \cite{Na}, it must be isomorphic to the nilpotent cone $\mathcal{N}(\mathfrak{k})$ of some complex Lie algebra $\mathfrak{k}$. The dimension constraint forces $\mathfrak{k} \simeq \mathfrak{sl}_2 \times \mathfrak{sl}_2$. 
However, the  Hilbert-Poincar\'e series for $\mathbb{C}[ \overline{\mathbb{O}}]$ 
is $\frac{1+t}{(1-t)^4}$ while the one of $\mathbb{C} [\mathcal{N}(\mathfrak{k})]$ is $\frac{(1+t)^2}{(1-t)^4}$.
Hence the two $\mathbb{C}$-algebras are not isomorphic as graded algebras and the two varieties cannot be isomorphic as $\mathbb{C}^\times$-varieties, contradicting our assumption.}
\end{rem}

\medbreak

We now turn to semisimple elements in $\V$. 

\begin{lemma}\label{lem:semisimple-elements} \ 
\begin{enumerate}
\setlength{\itemsep}{4pt}
\item An element of $\Vb$ is semisimple if and only if it lies in $\gl_n \subseteq \Vb$ and is a semisimple element therein.
\item An element of $\V$ is semisimple if and only if it is a semisimple element of $\bigwedge^2\k^{2n}$, if and only if it is a semisimple operator on $\k^{2n}$ via the inclusion $\bigwedge^2 \k^{2n} \subseteq \gl_{2n}$. Its $G$-orbit is  represented by a matrix ${\rm diag}(D, D)$ for some diagonal matrix $D\in \gl_n$.
\end{enumerate}
\end{lemma} 
\begin{proof}
 (1) Let $d \colon \k^\times \to \G$ denote the  cocharacter $\lambda \mapsto \lambda \Id_n$. Then  $d(\lambda)(v,x)=(\lambda v, x)$ for all $\lambda \in \k^\times$ and $(v, x) \in \Vb$. Therefore, if  
 $(v,x) \in \Vb$ has a closed $\G$-orbit,  $(0, x) \in \G(v,x)$, hence $v=0$.  It is well-known that the semisimple elements of $\gl_n$ are those admitting closed $\G$-orbit.

(2) Now consider the pair $(\V, G)$ and let $\mathbb B$ be the sum of $\k^n=V\subseteq T^*V$ with the subspace of  $\bigwedge^2\k^{2n}\subseteq \gl_{2n}$ consisting of matrices of the form $\left(\begin{smallmatrix}A&B\\0&^t\!A\end{smallmatrix}\right)$ with $A$ upper triangular and $B$ skew-symmetric. By \cite[Lemma 1.6]{springer}  every $G$-orbit meets $\mathbb B$. 
Assume that an element $(v,x)\in \mathbb B$ has a closed orbit.  Let  $x=\left(\begin{smallmatrix}A&B\\0&^t\!A\end{smallmatrix}\right)$ and  let $d' \colon \k^\times \to \G$ denote the  cocharacter $\lambda \mapsto {\rm diag}(\lambda \Id_n,\lambda^{-1}\Id_n)$. Then $d'(\lambda)(v,x)=\left(\lambda v,\left(\begin{smallmatrix}A&\lambda^2B\\0&^t\!A\end{smallmatrix}\right)\right)$, so $v=0$ and $B=0$. Hence $G(v,x)\in \bigwedge^2\k^{2n}$ and  \cite[Lemma 1.4]{springer} gives semisimplicity of $x\in\gl_{2n}$, whence of $A$ in $\gl_n$. Then $A$ is $\G$-conjugate to a diagonal matrix $D$, whence $x={\rm diag}(A, \,^t\!A)$ is $G$-conjugate to a diagonal matrix ${\rm diag}(D, \,D)$.
\end{proof}

\medbreak

Stabilisers of semisimple elements in $\Vb$ are thus Levi subgroups of $\G$.   Semisimple $G$-orbits in  $\bigwedge^2\k^{2n}$ are represented by diagonal matrices of the form 
\begin{equation}
\label{eq:ss-exotic}
{\rm diag}(a_1\Id_{\lb_1},\cdots,a_l\Id_{\lb_l},a_1\Id_{\lb_1},\cdots,a_l\Id_{\lb_l})
\end{equation}
where $a_i\in\k$ for all $i$ and $\lb=(\lb_1,\,\ldots,\,\lb_l)$ is a partition of $n$.
Stabilisers of semisimple elements in $\bigwedge^2\k^{2n}$ were computed in \cite{springer} and they are isomorphic to direct products of symplectic groups. More precisely, if $x$ is as in \eqref{eq:ss-exotic}
with $a_i\neq a_j$ for $1\leq i\neq j\leq l$, then $G^x\simeq G_{\lb}$ is the subgroup of $G$ consisting of matrices of the  form \begin{equation}\label{eq:centra}\left(\begin{smallmatrix}A_1&&&B_1&&\\
&\ddots&&&\ddots&\\
&&A_l&&&B_l\\
C_1&&&D_1&&\\
&\ddots&&&\ddots&\\
&&C_l&&&D_l\end{smallmatrix}\right)\end{equation} where $A_i,\,B_i, C_i$ and $D_i$ are $\lb_i\times \lb_i$ matrices,  and $\left(\begin{smallmatrix}A_i&B_i\\
C_i&D_i\end{smallmatrix}\right)\in \Sp_{2\lb_i}$ for any $i=1,\,\ldots,\,l$.  

\begin{remark}\label{rem:normalisers} \begin{enumerate}
\item Let $\mathbf x_s={\rm diag}(a_1\Id_{\lb_1},\cdots,a_l\Id_{\lb_l})\in \Vb$ with $a_i\neq a_j$ and let  $x_s=\varphi(\mathbf x_s)$. Any element in $N_{G}(G^{x_s})=N_{G}(\Sp_{2\lb})$ is obtained by composing an element of $G^{x_s}=\Sp_{2\lb}$ with an element of $N_G(T)$ whose image in the Weyl group $N_G(T)/T$ lies in the maximal standard parabolic subgroup  of $W$ of type ${\sf A}_n$. In other words, it corresponds to an element in the normaliser of the standard parabolic subgroup $\mathbb S_{\lb}$ of $\mathbb S_n$ associated with the partition $\lb$.

Let  $d_i(\lb):=\left|\{j\in\{1,\,\ldots,\,l\}, ~|~\lb_j=i\}\right|$, then   
 \begin{equation}\label{eq:normalisers}N_{G}(G^{x_s})/G^{x_s}\simeq \prod_{i=1}^n\mathbb S_{d_i(\lb)}\simeq N_{\mathbb S_n}(\mathbb S_\lb)/\mathbb S_\lb\simeq N_{\G}(\G^{\mathbf x_s})/\G^{\mathbf x_s}\end{equation}
where the composite  isomorphism $N_{\G}(\G^{\mathbf x_s})/\G^{\mathbf x_s}\simeq N_{G}(G^{x_s})/G^{x_s}$ is induced by the natural inclusion $\G\into G$. 
\item Two semisimple elements $x,\,x'\in\bigwedge^2\k^{2n}\subseteq\gl_{2n}$ are $\GL_{2n}$-conjugate if and only if they are $G$-conjugate. Indeed, by Lemma \ref{lem:semisimple-elements}(2) we may  assume that $x={\rm diag}(D,D)$ and $x'={\rm diag}(D',D')$ for some  diagonal matrices $D,\,D'\in\gl_n$. But then  $\GL_{2n}x_s=\GL_{2n}x'_s$ if and only if  $\G D=\G D'$, and the embedding $\G\into G$ gives  $G x_s=Gx_s'$.
\end{enumerate}
\end{remark}

\subsection{The Jordan--Kac--Vinberg decomposition}
\label{ss:JKVdecomposition}

Let $W$ be a rational module of a reductive group $H$. In Section~\ref{ss:nilpandss} we explained what it means for an element of $W$ to be semisimple or nilpotent.
Following \cite{GV} we say that $W$ admits a Jordan--Kac--Vinberg (JKV for short) decomposition if every element $w$ can be decomposed into a sum $w = w_s + w_n$ such that:
\begin{itemize}
\item $w_s$ is semisimple;
\item $w_n$ is nilpotent viewing $W$ as a $H^{w_s}$-module;
\item $H^w \subseteq H^{w_s}$.
\end{itemize}
Existence of a JKV decomposition for any representation $W$ is proved in \cite[Appendix]{Kac} by Kac using Luna's fundamental Lemma under the assumption that $\k$ has characteristic zero. In positive characteristics one may repeat the argument invoking \cite{BR}. When $W = \Lie(H)$ is the adjoint representation of a reductive group, 
these notions coincide with the usual abstract Jordan--Chevalley decomposition, and there is a unique such decomposition. When we write $y=y_s+y_n$ for some $y\in \gl_n$ or $\gl_{2n}$ we will always mean that $y_s$ and $y_n$ are, respectively, the semisimple and the nilpotent parts of $y$ in the Jordan-Chevalley decomposition. 

\smallskip

The Jordan decomposition in $\gl_n$ induces a JKV decomposition on $\Vb$ that is compatible with the projection onto $\gl_n$ as we now show.  Let $(v,x)\in \Vb$, and let $x=x_s+x_n$. Then $(v,x)=(0, x_s)+(v, x_n) $ is a JKV decomposition of $(v,x)$ in $\Vb$. Indeed, $(0, x_s)$ is semisimple. Its stabiliser is a Levi subgroup $L$ of $\G$, with Lie algebra $\l$. Then $L$ is a product of groups isomorphic to $\GL_{\lb_i}$ where $\lb$ is a partition of $n$ and $\k^n$ decomposes as an $L$-module as a direct sum of subspaces isomorphic to $\k^{\lb_1},\,\ldots,\,\k^{\lb_l}$ where the component $\GL_{\lb_i}$ acts on $\k^{\lb_i}$ via the standard representation and trivially on the other summands. Hence $\Vb_L:=\k^n\oplus \mathfrak l$ is the external direct sum of enhanced modules for the components of $L$. Since $x_n\in \mathcal N(\mathfrak l)$ we have $(v, x_n)\in \mathcal N(\Vb_L)$ by  \cite[Corollary~2.3]{springer}. Finally, $\G^{(v,x)}=\G^v\cap\G^x=\G^v\cap\G^{x_s}\cap\G^{x_n}\subset \G^{x_s}$. 

\smallskip

Similarly, the Jordan decomposition in $\gl_{2n}$ induces a JKV decomposition on $\V$ that is compatible with the morphism $\V\onto \bigwedge^{2}\k^{2n}\into\gl_{2n}$, as we now explain. Let $(v,x)\in \V$, so $x\in \gl_{2n}$ and let $x=x_s+x_n$.  We claim that $(v,x)=( 0, x_s)+(v, x_n) $ is a JKV decomposition of $(v,x)$ in $\V$.  First of all, by construction the endomorphisms $x_s,\,x_n\in\gl_{2n}$ are polynomials in $x$, and  powers of self-adjoint endomorphisms are again self-adjoint, so $x_s$ and $x_n$ lie in $\bigwedge^{2}\k^{2n}$.  Then,  $(0,x_s)$ is semisimple by \cite[Lemma 1.4]{springer} and $G^{x_s}\simeq \Sp_{2\lb}$ for some partition $\lb$ of $n$. In addition, $\k^{2n}$ decomposes as a $\Sp_{2\lb}$-module as a direct sum of subspaces isomorphic to $\k^{2\lb_1},\,\ldots,\,\k^{2\lb_l}$ where the component $\Sp_{2\lb_i}$ acts on $\k^{2\lb_i}$ via the standard representation and trivially on the other summands. 
Observe that $x_n\in \gl_{2n}^{x_s}\cap\bigwedge^2\k^{2n}$.  Up to $G$-action  $\gl_{2n}^{x_s}$ consists of matrices of the form
 \begin{equation}
 \label{eq:centraliser-double}
 \left(\begin{smallmatrix}A_1&&&B_1&&\\
&\ddots&&&\ddots&\\
&&A_l&&&B_l\\
C_1&&&D_1&&\\
&\ddots&&&\ddots&\\
&&C_l&&& D_l\end{smallmatrix}\right)
\end{equation}
where $A_i,\,B_i, C_i,\,D_i\in \gl_{\lb_i}$ for any $i=1,\,\ldots,\,l$, so 
$\gl_{2n}^{x_s}\cap\bigwedge^2\k^{2n}$ consists of matrices of the form
 \begin{equation}\label{eq:centraliser-exotic}\left(\begin{smallmatrix}A_1&&&B_1&&\\
&\ddots&&&\ddots&\\
&&A_l&&&B_l\\
C_1&&&^{t}\!A_1&&\\
&\ddots&&&\ddots&\\
&&C_l&&& ^{t}\!A_l\end{smallmatrix}\right)\end{equation} where $A_i,\,B_i$ and $C_i$ lie in $\gl_{\lb_i}$ and $B_i,\,C_i$ are skew-symmetric, for any $i=1,\,\ldots,\,l$.

Thus the $\Sp_{2\lb}$-submodule $\k^n\oplus (\gl_{2n}^{x_s}\cap\bigwedge^2\k^{2n})$ is the external direct sum $\V_{\lb}$ of  exotic modules for the components of $\Sp_{2\lb}$ and  $(v,x_n)\in \mathcal N(\V_{\lb})$, which is contained in the $G^{x_s}$-nilcone of $\V$. Finally, $G^{(v,x)}=G^v\cap G^x=G^v\cap G^{x_s}\cap G^{x_n}\subset G^{x_s}$.

\begin{remark}
\label{r:JKVnonunique}
The JKV  decomposition is not unique in general, see \cite[Appendix]{Kac}. It is pertinent to ask if it is unique in the enhanced or the exotic module, but it is not the case. Let  $n=2$, and  let $v=\left(\begin{smallmatrix}1\\1\end{smallmatrix}\right)$ and  $x=\left(\begin{smallmatrix}a&1\\0&b\end{smallmatrix}\right)$ with $a\neq b$. Then $x$ is semisimple and thus  $(v,x)=(0,x)+(v,0)$ is a JKV decomposition. However, we can also consider  the semisimple element $s=\left(\begin{smallmatrix}a&0\\0&b\end{smallmatrix}\right)$. Let  $\gamma\colon \k^\times\to \G^s$ be the cocharacter  given by $t\mapsto \left(\begin{smallmatrix}t^2&0\\0&t\end{smallmatrix}\right)$. Then $\gamma(t)(v,x-s)=\left( \left(\begin{smallmatrix}t^2\\t\end{smallmatrix}\right), \left(\begin{smallmatrix}0&t\\0&0\end{smallmatrix}\right)\right)$, so $0\in \overline{\G^s(v,x-s)}$. Finally, $\G^{(v,x)}=1\subseteq \G^{s}$. So, $(v,x)=(0,s)+(v,x-s)$ is also a JKV decomposition. A similar construction can be generalised to any $n$, and also applied to $\V$. 
\end{remark}

\medskip

The decompositions we considered are the unique ones that are compatible with the natural morphisms onto $\gl_n$ and $\k^n$, respectively to $\gl_{2n}$ and $\k^{2n}$. Observe that these Jordan decompositions are also compatible with the embeddings  \eqref{eq:inclusions}.
From now on we adopt this decomposition of elements in $\Vb$ and $\V$ as  JKV decomposition and whenever we write $(v,x)=(v,x_n)+(0,x_s)$ for an element in $\Vb$ or $\V$ we will always be referring to that. 

\medskip

The following result generalises Lemma~\ref{lem:poset-nil} to the orbit sets $\Vb/\G$,  $\V/G$ and $\Vb_{2n}/\GL_{2n}$, each of which is partially ordered by inclusion of closures, as usual. 

\begin{theorem}
\label{prop:double-dim}
\begin{enumerate}
\setlength{\itemsep}{4pt}
\item The morphism $\varphi$ is surjective with generic fibre of dimension $n$.  In particular,  every $G$-orbit in $\V$ meets $\Vb$.  
\item The morphism $\varphi$ induces an isomorphism of the (infinite) posets  $\Vb/\G$  and  $\V/G$.
\item For any $(v,x)\in \Vb$ we have $\dim G(v,x)=2\dim \G (v,x)$.  
 \item For any $(v,x)\in \V$ we have  $\dim \GL_{2n}(v, x)=2\dim G(v, x)$. 
 \item For any $X\subset \Vb$ and any  $Y\subset \V$ we have 
 \begin{equation}
 \label{eq:regularaftersaturation1}
 \varphi(1*X)^{\Greg}=\varphi(1*X^{\Gbreg}),\quad\quad\varphi(G\times^\G X)^{\Greg}=\varphi(G \times^\G X^{\Gbreg})
 \end{equation}
 and
 \begin{equation}
\label{eq:regularaftersaturation2} 
 \psi(1*Y)^{\GL_{2n}\operatorname{-reg}}=\psi(1*Y^{\Greg})\quad\quad\psi(\GL_{2n} \times^G Y)^{\GL_{2n}\operatorname{-reg}}=\psi(\GL_{2n} \times^G \ Y^{\Greg})
 \end{equation}
\item The morphism $\psi$ induces an injective poset morphism. 
 \end{enumerate}
\end{theorem}
\begin{proof}(1) Let $(v,x)=(0,x_s)+(v,x_n)\in\V$.  
Up to $G$-action,  we may assume that $x_s$ is of the form \eqref{eq:ss-exotic} and that $G^{x_s}=G_\lb$ for some partition $\lb=(\lb_1,\,\ldots,\,\lb_l)$ of $n$.  Since $x_n$ commutes with $x_s$ in $\gl_{2n}$, it has the form \eqref{eq:centraliser-exotic} and since it is also nilpotent,  the element  $(v,x_n)$ lies in the exotic nilpotent cone of $G_{\lb}$, which is a product of exotic nilcones.  By \cite[Theorem 6.1]{ah} applied to each factor $\Sp_{\lb_i}$ of $G_\lb$, the $\G_{\lb}$-orbit of  $(v,x_n)$ meets  the enhanced nilcone $\mathcal N(\Vb_\lb)$ of $\GL_{\lb}:=\G\cap G_\lb$.  As $\mathcal N(\Vb_\lb)\subset \Vb$, this gives surjectivity of $\varphi$.  The dimension of the generic fibre follows from a straightforward computation.

(2)  It follows from (1) that $\varphi$ induces an order preserving surjective map from $\Vb/\G$  to  $\V/G$: if $(v,x)\in \overline{\G(v',x')}$ then $(v,x)\in \overline{G(v',x')}$ so $\G(v,x)\subseteq\overline{\G(v',x')}$ implies $G(v,x)\subseteq\overline{G(v',x')}$.
  We now show that the induced map is injective. 
Assume $G(v,x)=G(v',x')$, for $(v,x),(v',x')\in \Vb$ and let $(v,x)=(0,x_s)+(v,x_n)$ and $(v',x')=(0,x'_s)+(v',x'_n)$. Lemma \ref{lem:semisimple-elements}(2) gives $x_s,\,x_s'\in\bigwedge^2\k^{2n}$. As $Gx=Gx'$, it follows that $\GL_{2n}x=\GL_{2n}x'$   whence $\GL_{2n}x_s=\GL_{2n}x'_s$. By Remark \ref{rem:normalisers}  we may thus assume $x_s=x_s'$.
Then, $(v,x_s+x_n)$ and $(v',x_s+x_n')$ are $G^{x_s}$-conjugate and  $(v,x_n)$ and $(v',x_n')$ lie in  the enhanced nilcone for $\G^{x_s}$. Thus, they are $\G^{x_s}$-conjugate thanks to Lemma~ \ref{lem:poset-nil}, giving injectivity of the induced map.  

We now show that the map is order reflecting. It is more convenient to differentiate between elements in $\Vb$ and their images in $\V$. Recall that we write $\varphi(v,x)$ for $\varphi(1*(v,x))$ and similarly for $\psi$. 
Let $(v,x)=(v,x_s+x_n)$ and $(v',x') = (v',x_s'+x_n')\in \Vb$ and assume $\varphi(v,x)\in\overline{G\varphi(v',x')}$. Then, $\psi(\varphi(v,x))\in\overline{\GL_{2n}\psi\varphi(v',x')}$, so $\psi\varphi(x_s)$ and $\psi\varphi(x'_s)$ are $\GL_{2n}$-conjugate and  $\varphi(x_s), \varphi(x'_s)\in\bigwedge\!\!^2\k^{2n}$. By Remark \ref{rem:normalisers} (2),  we may assume that $x_s=x_s'$
 and that $\varphi(0,x_s)$ is of the form \eqref{eq:ss-exotic}. Then 
\begin{align*}
\varphi(v,x_s+x_n)&\in \overline{G\varphi(v',x_s+x_n')}\cap \left(\varphi(0,x_s)+\varphi(\mathcal N(\Vb_\lb)\right)\\
&\subseteq \varphi(0,x_s)+\left(\overline{G_\lb\varphi(v',x_n')}\cap \varphi(\mathcal N(\Vb_\lb))\right).
\end{align*}
Lemma~\ref{lem:poset-nil} (1) applied to $\Vb_\lb$ and $\V_\lb$ gives 
\begin{equation*}\varphi(v,x)\in \varphi(0,x_s)+\overline{\varphi(\G_\lb(v',x_n'))}=\varphi\left(\overline{\G_{\lb}(v',x')}\right),\end{equation*}
whence $(v,x)\in \overline{\G_{\lb}(v',x')}\subset\overline{\G(v',x')}$. 

(3) First of all for $x_s$ as in \eqref{eq:ss-exotic} \[\dim Gx_s=2n^2+n-\sum_{i=1}^l(2\lb_i^2+\lb_i)=2\dim \G x_s.\]
In addition, $\G^{x_s}\simeq\prod_{i=1}^l\GL_{\lb_i}\leq \prod_{i=1}^l\Sp_{2\lb_i}\simeq G^{x_s}$ with compatible inclusions of enhanced and exotic nilcones. Combining with 
 Lemma \ref{lem:poset-nil} gives
\begin{align*}
\dim G(v, x)&=\dim G{x_s}+\dim G^{x_s}(v,x_n)=2\dim \G x_s+2\dim \G^{x_s}(v,x_n)=2\dim \G(v, x).
\end{align*}
\medskip
(4)  Let $(v,x)\in\Vb$ be as above.  Then,  $\GL_{2n}^{x_s}$ consists of invertible matrices of the form \eqref{eq:centraliser-double}, and so $\GL_{2n}^{x_s}\simeq\prod_{i=1}^l\GL_{2\lb_i}$.  
Hence,
\begin{align}\label{eq:aux-ss}\dim \GL_{2n}x_s&=4n^2-\sum_{i=1}^l(2\lb_i)^2=2(2n^2-n+\sum_{i=1}^l(2\lb_i^2-\lb_i))\\
\nonumber &=2(\dim G-\dim G^{x_s})=2\dim Gx_s\end{align}
Now $(v, x_n)$ lives in the exotic  nilcone of $\Sp_\lb$ which is a product of exotic nilcones,  we denote by $(v(i), x_n(i))$ the $i$-th component.  
Then,  making use of \eqref{eq:aux-ss} and Lemma \ref{lem:poset-nil} we have:
\begin{align*}\dim \GL_{2n}(v,x)&=\dim \GL_{2n}x_s+\dim \GL_{2n}^{x_s}(v,x_n)\\
&=2\dim G x_s+\sum_{i=1}^l\dim \GL_{2\lb_i}(v(i),x_n(i))\\
&=2\dim Gx_s+2\sum_{i=1}^l\dim \Sp_{2\lb_i}(v(i),x(i))\\
&=2\left(\dim Gx_s+\dim G^{x_s}(v,x_n)\right)=2\dim G(v,x).\end{align*}

(5) The two equalities in \eqref{eq:regularaftersaturation1} are equivalent, as are the two equalities in \eqref{eq:regularaftersaturation2}. All four equalities readily follows from the dimension formulas (3) and (4). 

(6) We first show injectivity. 
 Let $(v,x_s+x_n),(v', x'_s+x_n')\in \V$ be such that $\psi(v', x_s'+x_n)\in \GL_{2n}\psi(v,x_s+x_n)$. 
 Since  $\psi(0,x_s)\in \GL_{2n}\psi(0,x_s')$,  combining Lemma \ref{lem:semisimple-elements}(2) and Remark \ref{rem:normalisers} we may assume that $x_s'=x_s$.   
  Whence,   $(v,x_n)$ and $(v', x_n')$ lie in the exotic nilcone of $G^{x_s}$ and their images in $\Vb_{2n}$  are $\GL_{2n}^{x_s}$-conjugate. We conclude applying Lemma \ref{lem:poset-nil}. Preservation of closures of inclusions is proved as in  (2).
\end{proof}

We consider also the natural projections 
\begin{align}\label{eq:projections}p\colon \V\to \bigwedge{}\!\!^2 \k^{2n},&&p'\colon \Vb\to \gl_n.\end{align} By general principles both morphisms are open. As we now show, they are also well-behaved with respect to closed invariant subsets.

\begin{cor}\label{lem:closure-proj}
If $X$ is a $G$-stable subset of $\V$ then $p(\overline{X})=\overline{p(X)}$. Similarly, if $Y$ is a $\G$-stable subset of $\Vb$ then $p'(\overline{Y}) = \overline{p'(Y)}$.
\end{cor}
\begin{proof}
By continuity of the projections $p(\overline{X})\subseteq  \overline{p(X)}$ and $p'(\ov Y) \subseteq \ov{p'(Y)}$. Since $p(X) \subseteq p(\ov X)$ and $p'(Y) \subseteq p'(\ov Y)$ we can complete the proof by showing that $p(\ov X)$ and $p'(\ov Y)$ are closed.

Now let $d \colon \k^\times \to \G$ denote the  cocharacter $\lambda \mapsto \lambda \Id_n$.
 Then $d(\lambda)  (v,x)=(\lambda v, x)$ for all $\lambda \in \k^\times$ and $(v, x) \in \Vb$. If $C \subseteq \Vb$ is $\G$-stable and closed then $(0, x) \in C$ for all $(v, x) \in C$. It follows that $p'(C) = p'(C \cap (0,\gl_n))$, which is closed because the restriction of $p'$ to $(0,\gl_n)$ is trivially an isomorphism. Setting $C=\overline{X}$ completes the proof for $p'$.

If $C $ is a closed $G$-invariant subset of $\V$ then for $(v,x) \in C$ there is some $G$-conjugate in $\Vb$, by Theorem~\ref{prop:double-dim}(1). By the previous paragraph we see that $(0,x) \in C$, and so $p(C) = p(C \cap (0,\bigwedge^2 \k^{2n}))$, and hence $p(C)$ is closed. Setting $C=\overline{Y}$ completes the proof. 
\end{proof}

\subsection{Categorical quotients of the exotic and enhanced modules} The rings of invariants $\k[\Vb]^{\G}$ and $\k[\V]^G$ are described in \cite[Propositions 1.5,1.7\&2.2]{springer}. It is shown that $\k[\Vb]=\k[\gl_n]^{\G}$, so $\Vb/\!/\G=\Spec \k[\Vb]^\G\simeq {\Lie}(\T)/\mathbb S_n\simeq \mathbb{A}^n_\k$ is an affine space. Similarly, it is shown that
$\k[\V]^G=\k[\bigwedge\!\!^2\k^{2n}]^G$ and that $\V/\!/G = \Spec \k[\V]^G \simeq \mathbb{A}^n_\k$. 

We consider the quotient maps
\begin{eqnarray}
\begin{array}{rcl}
\label{eq:quotientmaps}
\pi_{\V,G} & : & \V \longrightarrow \V/\! /G\\
\pi_{\Vb,\G} & : & \Vb \longrightarrow \Vb /\!/ \G.
\end{array}
\end{eqnarray}
\begin{theorem}
\label{T:quotients}
\begin{enumerate}
\setlength{\itemsep}{4pt}
\item The inclusion $\Vb\into \V$ induces an isomorphism $\V/\!/ G \simeq\Vb /\!/ \G$.
\item The morphisms $\pi_{\V,G}$ and $\pi_{\Vb,\G}$ are flat and surjective, hence faithfully flat.
\item The fibres of both morphisms consist of finitely many orbits, and each fibre contains a unique closed orbit. 
\item The fibres of $\pi_{\Vb,\G} $ are irreducible, normal complete intersections.
\item The fibres of $\pi_{\V,G}$ are irreducible, normal complete intersections.
\end{enumerate}
\end{theorem}
\begin{proof}
By \cite[Proposition 1.5]{springer}, the restriction map $\k[\V]^G\to\k[\V^T]^{N_G(T)/T}$ is injective. Through the inclusion $\Vb\into\V$, the subspace $\V^T$ is identified with ${\rm Lie}(\T)$. 
Remark \ref{rem:normalisers} applied to a regular semisimple element in ${\rm Lie}(\T)$ shows that  ${\rm Lie}(\T)/\mathbb S_n= \V^T/(N_G(T)/T)$. 
Combining with Chevalley restriction theorem we obtain an injective graded algebra morphism $\k[\V]^G\into \k[\Vb]^\G$.  By \cite[\textsection 1.2 \& Proposition 1.7]{springer}, the generators of $\k[\V]^G$ are algebraically independent and have degrees $1,2,...,n$.
Comparing  with the degrees of the generators of $\k[\Vb]^{\G}=\k[\gl_n]^{\G}$ gives (1).

Now we prove (2) for $\pi_{\V,G}$. Since $\k[\V]$ is free over $\k[\bigwedge^2 \k^{2n}] \otimes 1$, it suffices to show that $\bigwedge^2 \k^{2n} \to \bigwedge^2 \k^{2n} /\!/ G$ is flat and surjective, whilst analogous remarks hold for the enhanced module.

The map $\bigwedge^2 \k^{2n} \to \bigwedge^2 \k^{2n} /\!/ G$ is faithfully flat when $\k =\C$ thanks to \cite[Theorem~15]{KR71}, whilst when $\chr(\k) > 0$ it can be deduced from \cite[Corollary~2.22]{Le09}, using miracle flatness. This completes the proof of (2) for $\pi_{\V,G}$. The argument for $\pi_{\Vb,\G} $ is identical, using \cite[Proposition~7.13]{JaNO} to see that $\gl_n \to \gl_n /\!/ \G$ is faithfully flat, when $\k$ has any characteristic. This completes (2).

Thanks to \cite[Proposition~2.6]{springer} there are finitely many $\G$-orbits in $\cN(\Vb)$. Combining with Lemma~\ref{lem:poset-nil}(1) the same holds for $G$-orbits in $\cN(\V)$. Applying Corollary~2 of \cite[\textsection 5.2]{PV94} and \cite[\S 2.1.3]{BR} we deduce (3). Note that the cited corollary is formulated over $\C$, but the argument applies equally well to any algebraically closed field.

By (1) of the current theorem, we know that $\Vb/\!/ \G \simeq \gl_n /\! / \G$. Furthermore, for $s \in \gl_n /\!/ \G$ and writing $\pi_{\gl_n,\G} : \gl_n \to \gl_n /\! / \G$ for the adjoint quotient, we have that $(\pi_{\Vb,\G} )^{-1}(s) \simeq \k^n \times \pi_{\gl_n,\G}^{-1}(s)$ as reduced schemes. By \cite[Propositions 7.13 and 8.5]{JaNO}, the fibres of $\pi_{\gl_n,\G}$ are irreducible, normal complete intersections. This proves (4).


It remains to prove (5). The argument for normality in the present paragraph is quite standard, and is modelled on the adjoint representation (see \cite[Proposition~8.5]{JaNO} for background and references). Since $\pi_{\V,G}$ is faithfully flat the fibres are non-empty and all have the same dimension, which is $\dim \V - n$. Since each fibre is defined by precisely $n$ polynomial functions, they are complete intersections, hence Cohen--Macaulay. By part (3) and Theorem~\ref{prop:double-dim} we see that each fibre is a union of finitely many orbits of even dimension, hence each fibre is regular in codimension 2. We can now apply Serre's criteria to see that the fibres are normal.

Finally we show that the fibres of $\pi_{\V,G}$ are irreducible, starting with the central fibre. Since $\cN(\V)$ is $\k^\times$-stable the irreducible components intersect at $0 \in \cN(\V)$, but normality ensures the components cannot intersect. Therefore $\cN(\V)$ is irreducible. Let $x \in \V /\!/G$. The natural grading on $\k[\V]$ induces a filtration on $\k[\V]$ that descends to a filtration on $\k[\pi_{\V,G}^{-1}(x)]$. Passing to its associated graded we have $\k[\cN(\V)] \twoheadrightarrow \gr \k[\pi_{\V,G}^{-1}(x)]$ as graded algebras. Since the fibres of $\pi_{\V,G}$ all have the same dimension and $\pi_{\V,G}^{-1}(0)$ is reduced, it follows that $\k[\cN(\V)] \cong \gr \k[\pi_{\V,G}^{-1}(x)]$. Now $\k[\pi_{\V,G}^{-1}(x)]$ is an integral domain, and so $\pi_{\V,G}^{-1}(x)$ is irreducible. This concludes (5).
\end{proof}

\begin{remark}\label{rem:saturation}{\rm Let $X\subset\Vb$ be  $\pi_{\Vb,\G}$-saturated. The set $\widetilde{X}=\pi_{\V,G}^{-1}\pi_{\Vb,\G}(X)$ defined using the identification in Theorem \ref{T:quotients} (1)  is $\pi_{\V,G}$-saturated. If $X$ is also open, then $\widetilde{X}$ is open. }
\end{remark}

\section{Enhanced and Exotic Jordan classes}\label{ss:EnhancedExoticClasses}


We introduce equivalence relations on $\Vb$ and $\V$. The elements $(v,x), (v',x') \in\Vb$ with JKV decomposition $(v,x)=(0,x_s)+(v,x_n)$ and  $(v',x')=(0,x'_s)+(v',x'_n)$ are said to be Jordan equivalent if there is $g\in \G$ such that 
$g \G^{x_s} g^{-1}=\G^{x'_s}$,  and $g(v,x_n)=(v',x_n')$.  Equivalence classes with respect to this relation are called {\it enhanced Jordan classes}. 
Observe that the Levi subgroup $\G^{x_s}$ contains a maximal torus of $\G$, so $\G^{x_s}$ does not fix nonzero vectors in $\k^n$. 
This implies that 
$\Vb^{\G^{x_s}}\subset\mathfrak{gl}_n$ and therefore it equals $\mathfrak{z}(\mathfrak{gl}_n^{x_s})$.
Thus  the enhanced  Jordan class of  $(v,x)$ is  the set

\begin{equation}
\label{eq:enhanced-jordan-class}
\J(v,x)=\G(v,(\Vb^{\G^{x_s}})^{\Gbreg}+x_n)=
\G\left(v,\mathfrak{z}(\gl_n^{x_s})^{\Gbreg}+x_n\right).
\end{equation}

If $v=0$, then any element equivalent to $(v,x)$ has trivial component in $\k^n$,  and its equivalence class is a usual Jordan class in $\mathfrak{gl}_n$.  More generally, the projection through $p$ as in \eqref{eq:projections}  of any enhanced Jordan class $\mathfrak{J}$ is a usual Jordan class in $\mathfrak{gl}_n$.

\medskip

The elements $(v,x), (v',x') \in\V$ with JKV decomposition $(v,x)=(0,x_s)+(v,x_n)$ and  $(v',x')=(0,x'_s)+(v',x'_n)$ are said to be Jordan equivalent if there is $g\in G$ such that 
$g G^{x_s} g^{-1}=G^{x'_s}$,  and $g(v,x_n)=(v',x_n')$.  Equivalence classes with respect to this relation are called {\it exotic Jordan classes}. 

Since $G^{x_s}$ contains a maximal torus of $G$, which can never fix nonzero vectors in $\k^{2n}$, we have  
$\V^{G^{x_s}}=\{(0,x')\in\V~|~G^{x_s}\subset G^{(0,x')}\}$ and $(\V^{G^{x_s}})^{\Greg}=\{(0,x')\in\V~|~G^{x_s}= G^{(0,x')}\}$ contains only elements with reductive centralisers. Therefore \cite[Corollary 4.7]{springer} implies that $(\V^{G^{x_s}})^{\Greg}$ consists of semisimple elements. The exotic Jordan class of $(v,x)$ is the set 
\begin{equation}\label{eq:exotic-jordan-class}J(v,x)=G (v,(\V^{G^{x_s}})^{\Greg}+x_n)=G \{(v,x'_s+x_n)\in\V~|~G^{x_s}= G^{x_s'}\}.
\end{equation}

\begin{remark}\label{rem:reg-ss}Let $\mathbf x_s:= {\rm diag}(a_1\id_{\lb_1},\,\ldots,\,a_l\id_{\lb_l})\in \gl_n$ with $a_i\neq a_j$ for $i\neq j$, and let $(0,x_s)=\varphi(0,\mathbf x_s)$.
The fixed points subspace $\V^{G^{x_s}} = \V^{G_\lb}$ coincides with the set of diagonal matrices of the form \eqref{eq:centraliser-exotic}, with fixed $\lb$. It also coincides with $(0, \mathfrak z(\gl_{2n}^{x_s}))$ under the inclusion $\V = \psi(\V) \subseteq \Vb_{2n}$. With the help of Proposition \ref{prop:double-dim} we see that the three open subsets 
\begin{equation*} 
(0,\mathfrak{z}(\gl_n^{\mathbf  x_s})^{\Gbreg})\subset \Vb, \quad(\V^{G^{x_s}})^{\Greg}\subset \V\quad \mbox{ and } \quad(0,\mathfrak{z}(\gl_{2n}^{x_s})^{\GL_{2n}\operatorname{-reg}})\subset\Vb_{2n}\end{equation*} 
are  identified through $\varphi$ and $\psi$.
\end{remark}

\begin{proposition}\label{prop:basic-properties}
\begin{enumerate}
\item There are finitely many enhanced  Jordan classes, and each one is irreducible, $\G$-stable, and consists of $\G$-orbits of fixed dimension.
They are  parametrized by $\G$-orbits of pairs $(L,\bO)$ where $L$ is a Levi subgroup of $\G$  and $\bO$ is an $L$-orbit in the enhanced nilcone of $L$.
  Equivalently they are parametrized by pairs $(\lb,\bO)$ where $\lb$ is a partition of $n$ and $\bO$ is an $N_{\G}(\G_{\lb})$-orbit in the enhanced nilcone of $\G_{\lb}$. 
\item There are finitely many exotic  Jordan classes, and each one is irreducible, $G$-stable, and consists of $G$-orbits of fixed dimension.
They are parametrized by $G$-orbits of pairs $(M,\O)$ where $M$ is the stabiliser of a semisimple element in $\V$ and $\O$ is an $M$-orbit in the exotic nilcone of $M$.  
 Hence, they are parametrized by pairs $(\lb,\O)$ where $\lb$ is a partition of $n$ and $\O$ is an $N_G(G_{\lb})$-orbit in the exotic nilcone of  $G_{\lb}$.
\item For each $m\in \mathbb N$ choose a total ordering $\leq$ on the set $\mathcal{Q}_m$ of bipartitions of $m$. The set of  enhanced  Jordan classes and the set of exotic  Jordan classes are both in bijection with the set of pairs $(\lb, (\mu  ; \nu))$ where $\lb=(\lb_1,\,\ldots,\,\lb_l)$ is a partition of $n$ and  $(\mu ; \nu)= ((\mu_1;\nu_1),\,\ldots,\,(\mu_l;\nu_l))$ is an $l$-tuple where the $i$-th component is a bipartition of $\lb_i$ and if $\lb_i=\lb_{i+1}$ then $(\mu_i;\nu_i)\geq(\mu_{i+1};\nu_{i+1})$. 
\end{enumerate}
\end{proposition}
\begin{proof}
By construction enhanced and exotic Jordan classes are unions of orbits of the same dimension.  Irreducibility follows because  the enhanced, respectively exotic Jordan class of $(0,x_s)+(v,x_n)$ is the   image through the action morphism of $\G$, respectively  $G$, of the  irreducible variety $(v,\mathfrak z(\gl_n^{x_s})^{\Gbreg}+x_n)$, respectively 
$(v,(\V^{G^{x_s}})^{\Greg}+x_n)$.  The parametrisation  in terms of pairs follows from the construction.   By the description of the centralisers of semisimple elements, 
 the set of orbits of stabilisers of semisimple elements is, in both cases,  in bijection with the set of partitions $\lb$ of $n$,  and the set of enhanced, respectively exotic nilpotent orbits for $\GL_{\lb}$, respectively, $G_{\lb}$  is finite \cite{ah,K1,springer}, therefore the set of orbits in the nilpotent cone for the normalisers of the centralisers of semisimple elements are a fortiori finite.  Statement (3) follows combining (1) and (2) with Remark \ref{rem:normalisers}. 
\end{proof}

Observe that nilpotent orbits are examples of Jordan classes. The following result is a generalisation  of Lemma \ref{lem:poset-nil}. We use the inclusions $\Vb \subseteq \V \subseteq \Vb_{2n}$ from \eqref{eq:towerofinclusions} in what follows.
\begin{theorem}\label{thm:main-bijection}
\begin{enumerate}
\setlength{\itemsep}{4pt}
\item The assignment $\mathfrak J\mapsto \varphi(G\times^{\G}\mathfrak J)$   establishes a bijection between the set of enhanced Jordan classes and the set of exotic Jordan classes, with inverse induced by  $J\mapsto J\cap \Vb$.
\item The assignment $J\mapsto\psi(\GL_{2n} \times^G J)$ establishes an injective map from the set of exotic Jordan classes in $\V$ to the set of enhanced Jordan classes in $\Vb_{2n}$. The assignment $\J\mapsto \J\cap \V$ is a left inverse. 
\end{enumerate}
\end{theorem}
\begin{proof}
(1)  Combining the description of the enhanced and the exotic Jordan classes \eqref{eq:enhanced-jordan-class}, \eqref{eq:exotic-jordan-class} with Remark  \ref{rem:reg-ss} we see that if $\mathfrak J$ is the enhanced Jordan class of $(v,x)$, then
$G\J=\varphi(G\times^{\G}\J)$ is the exotic Jordan class of $\varphi(v,x)$, so the map is well-defined. It is surjective by Proposition \ref{prop:double-dim}. We prove injectivity. Let $(v,x)$ and $(v',x')\in\Vb$ have Jordan classes $\J$ and $\J'$, such that
$\varphi(G\times^{\G}\mathfrak J)=\varphi(G\times^{\G}\J')$. Then the classes $J$ and $J'$ of $\varphi(v,x)$ and $\varphi(v', x')$ in $\V$ are equal. Let $x_s$ and $x_s'$ be the semisimple parts of $x$ and $x'$, respectively. We may assume that $G^{\varphi(x_s)}=G^{\varphi(x_s')}$ so Remark~\ref{rem:reg-ss} shows that $\G^{x_s}=\G^{x'_s}$. Let $x_n$ and $x_n'$ be the nilpotent parts of $x$ and $x'$, respectively. Then, $x_n,x_n'\in\gl_n^{x_s}\cap\cN(\Vb)$, so $(v,x_n)$ and $(v',x_n')$ lie in the enhanced nilpotent cone for $\G^{x_s}$ and their images thrhough $\varphi$ are $G^{x_s}$-conjugate. Hence they are $\G^{x_s}$-conjugate by Lemma~\ref{lem:poset-nil}, and we conclude that $\J=\J'$.

 (2) Combining the description of the enhanced and the exotic Jordan classes \eqref{eq:enhanced-jordan-class}, \eqref{eq:exotic-jordan-class} with Remark  \ref{rem:reg-ss} we see that if $J$ is the exotic Jordan class of $(v,x)$, then
$\GL_{2n}\psi(J)=\psi(\GL_{2n}* J)$ is the exotic Jordan class of $\psi(v,x)$, so the map is well-defined.  Injectivity is proved as in (1), observing that all nilpotent orbits in the enhanced module $\Vb_{2n}$ that are in the image of $\psi$ correspond to bipartitions of the form $(\mu\cup\mu ; \nu\cup\nu)$.   
\end{proof}

\subsection{Local study of enhanced Jordan classes}\label{sec:geometryofclasses} 
In this Section we study the geometry of  enhanced Jordan classes, showing that they are locally closed and smooth.

Let $M$ be a Levi subgroup of a parabolic subgroup of $\G$ and let ${\mathfrak m}={\rm Lie}(M)$.
We set  
\begin{equation}\label{eq:mcirc}
{\mathfrak m}^\circ:=\{x\in\mathfrak{m}~|~\G^{x_s}\subseteq M\}.
\end{equation}
By \cite[Lemma 2.1]{broer}, whose proof holds also for the reductive group $\G$ and for arbitrary characteristic of $\k$, the subset $\mcirc$ is a $\pi_{\mathfrak{m},M}$-saturated open affine subset of $\mathfrak{m}$.  Hence
 we have
\begin{equation}
{\mathfrak m}^\circ=\{x\in\mathfrak{m}~|~\G^{x}\subseteq M\}.
\end{equation}
In addition,   $\Vb_M:=\k^n\oplus\mathfrak{m}$ is the enhanced module for $M$, and so   $\Vb_M/\!/M=\mathfrak{m}/\!/M$, \cite{springer},  hence the subset $(\k^n,\mcirc)$ is a $\pi_{\Vb_M,M}$-saturated open affine subset of $\Vb_M$. For any semisimple element $y\in\mathfrak{m}^\circ$  we have:
\begin{align}\label{eq:regM=regG}
\mathfrak{z}(\gl_n^y)^{\Gbreg}=\{z\in\mathfrak{gl}_n~|~\gl_n^y=\gl_n^z\}=\{z\in\mathfrak{gl}_n~|~\mathfrak{m}^{y}=\gl_n^z\}=\mathfrak{m}^\circ\cap\mathfrak{z}(\mathfrak{m}^y)^{M\operatorname{-reg}}.
\end{align}

\begin{lemma}\label{lem:slice-etale}
Let $M$ be a Levi subgroup of a parabolic subgroup of $\G$ with enhanced module $\Vb_M$,  let $\mathfrak{m}={\rm Lie}(M)$, let $y$ be a semisimple element in ${\mathfrak m}^\circ$, and let
\begin{equation}
\phi\colon \G\times^M\Vb_M\to \Vb, \quad g*(v,x)\mapsto g(v,x).
\end{equation}
Then,  there are a $\pi_{\Vb,\G}$-saturated principal open subset $U_{\mathfrak{gl_n}}$ of $\Vb$ and a $\pi_{\Vb_M,M}$-saturated principal open  subset $U_{\mathfrak m}$ of  $\Vb_M$ containing $(0,y)$,  such that the restriction of $\phi$ to $\G\times^MU_{\mathfrak m}$ is a surjective \'etale morphism to $U_{\mathfrak{gl}_n}$.
\end{lemma} 
\begin{proof}
We verify the hypotheses of \cite[Theorem 6.2]{BR}, that is,  the arbitrary characteristic extension of Luna's fundamental Lemma.  The varieties $\G\times^M\Vb_M$ and $\Vb$ are smooth and irreducible.  The orbit $\G*(0,y)$ is closed because  it is the image through $\pi_{\G\times \Vb_M, M}$ of the closed subset $\G \times \{(0,y)\}$.  
In addition, 
$\phi(\G*(0,y))=\G (0,y)\subset (0,\mathfrak{gl}_n)$ is a semisimple adjoint orbit, whence it is closed.  Also, $\phi$ is injective on $\G*(0,y)$: indeed if $\phi(g*(0,y))=\phi(h*(0,y))$ for some $g,h\in\G$, then $m:=g^{-1}h\in G^y\subset M$,  so $h*(0,y)=gm*(0,y)=g*m(0,y)=g*(0,y)$.  Finally, we verify that $\phi$ is \'etale at $1*(0,y)$.

We define the maps 
\begin{align*}
\iota_1&\colon \G\to \G\times\Vb_M, && g\mapsto (g,0,y),\\
\iota_2&\colon V\to  \G\times\Vb_M&& v\mapsto (1,v,y),\\
\iota_3&\colon \mathfrak m\to  \G\times\Vb_M&& b\mapsto (1,0,b).
\end{align*}
Then, by \cite[\S 16.1]{borel},  the tangent map at $(1,0,y)$ of the morphism $\widetilde{\phi}\colon \G\times \Vb_M\to \Vb$ given by 
$(g*(v,x) )\mapsto (gv,gxg^{-1})$ is the linear map $d_{(1,0,y)}\widetilde{\phi}\colon \mathfrak{gl}_n\oplus \Vb_M\to \k^n\oplus\mathfrak{gl}_n$ given by 
\begin{align*}
d_{(1,0,y)}(B,v,x)&=d_1(\widetilde{\phi}\circ\iota_1)(x)+d_0(\widetilde{\phi}\circ\iota_2)(v)+d_y(\widetilde{\phi}\circ\iota_3)(B)\\
&=(0+v+0,[x,y]+0+B)
\end{align*}
for $B\in \gl_n$ and $(v,x)\in \Vb_M$.
Surjectivity of the map $\mathfrak{gl}_n\times  {\mathfrak m}\to \mathfrak{gl}_n$ given by $(B,x)\mapsto [x,y]+B$ is proved as in the proof of \cite[Lemma 2.1]{broer}, which is still valid in arbitrary characteristic. Therefore the above tangent map is surjective at $(1,0,y)$, implying surjectivity
of the tangent map of $\phi$ at $1*(0,y)$. By dimensional reasons the latter is also bijective whence $\phi$ is \'etale at $1*(0,y)$ by  \cite[Proposition 2.9]{milne}.
By \cite[Theorem 6.2]{BR}  there exists $f \in k[\Vb]^{\G}$ such that $f(\G y)=1$ and for which the restriction of $\phi$ to the principal open subset 
$\left(\G\times^M\Vb_M\right)_{f\circ\phi}$  corresponding to $f\circ\phi\in \k[ \G\times^M\Vb_M]^{\G}$
 is a surjective \'etale morphism from $\left(\G\times^M\Vb_M\right)_{f\circ\phi}$ to the principal open subset $U_{\mathfrak{gl}_n}:=\Vb_f$ of $\Vb$ associated with $f$.  
 Now,  $f\circ\phi\in\k[ \G\times^M\Vb_M]^{\G}=\k[\Vb_M]^M$.  
 The statement follows taking   $U_{\mathfrak{m}}$  to be the corresponding principal open subset in $\Vb_M$,   so that  $\left(\G\times^M\Vb_M\right)_{f\circ\phi}=\G\times^MU_{\mathfrak{m}}$.
\end{proof}

\begin{remark}
{\rm It was already observed in \cite{broer2} that for the adjoint action of $\G$, the geometry of  ordinary Jordan (decomposition) classes and varieties is independent of the characteristic of $\k$.}
\end{remark}

\begin{lemma}\label{lem:intersection}Let $M$ be a Levi subgroup of a parabolic subgroup of  $\G$,  let $\mathfrak{m}={\rm Lie}(M)$ and let $\mathfrak{J}$ be an enhanced  Jordan class intersecting $(\k^n,\mathfrak{m}^\circ)$. 
Let $\mathfrak{J}_{M,1},\,\ldots,\,\mathfrak{J}_{M,r}$  be enhanced Jordan classes  for the $M$-action satisfying $\mathfrak{J}\cap \mathfrak{J}_{M,i}\cap (\k^n,\mathfrak{m}^\circ)\neq\emptyset$.

Then,  
\begin{enumerate}
\item $\mathfrak{J}\subset \G(\k^n+\mathfrak{m}^\circ)$.
\item $(\k^n,\mathfrak{m}^\circ)\cap \mathfrak{J}= (\k^n,\mathfrak{m}^\circ)\cap\left(\bigcup_{i=1}^r\mathfrak{J}_{M,i}\right)$ 
\item Let $(v,y)=(v,y_n)+(0,y_s)\in \mathfrak{J}$. Then 
\begin{equation*} 
(\k^n,(\gl_n^{y_s})^\circ)\cap\mathfrak{J}=(\k^n,(\gl_n^{y_s})^\circ)\cap \left((0,\mathfrak{z}(\gl_n^{y_s})^{\G^{y_s}\operatorname{-reg}})+N_{\G}(\G^{y_s})(v,y_n)\right).\end{equation*}
\end{enumerate}
\end{lemma}
\begin{proof}
(1) Let $(v,x)=(v,x_s+x_n)\in (\k^n,\mathfrak{m}^\circ)\cap\mathfrak{J}$. Then,  $\mathfrak{gl}_n^{x_s}\subseteq\mathfrak{m}$ and  $x_s\in \mathfrak{z}(\gl_n^{x_s})$, which is contained in every maximal torus of $\gl_n^{x_s}$ and therefore in some maximal torus of $\mathfrak{m}$. 
Hence, \eqref{eq:regM=regG} gives
$\mathfrak{z}(\gl_n^{x_s})^{\Gbreg}=\{z\in\mathfrak{gl}_n \mid \gl_n^{x_s}=\gl_n^z\}\subseteq\mathfrak{m}^\circ$. Since $\k^n,\mathfrak{m}^\circ$ is saturated, 
 $\mathfrak{J}\subset \G(\k^n,\mathfrak{m}^\circ)$. 
 
 (2) The inclusion $\subseteq$ is immediate from the definition of the $\mathfrak{J}_{M,i}$. We  prove $\supseteq$. Let  $\mathfrak{J}_{M,i}$ for some $i\in\{1,\ldots,\,r\}$. By assumption, 
 there is  $(v,x)=(v,x_s+x_n)\in \mathfrak{J}_{M,i}\cap\mathfrak{J}\cap (\k^n,\mathfrak{m}^ \circ)$. Using  $M$-stability of $(\k^n,\mathfrak{m}^ \circ)$ twice and \eqref{eq:regM=regG} we obtain:
\begin{align*}
 (\k^n,\mathfrak{m}^\circ)\cap\mathfrak{J}_{M,i}&= (\k^n,\mathfrak{m}^\circ)\cap\left(M\left(v,\mathfrak{z}(\mathfrak{m}^{x_s})^{M\operatorname{-reg}}+x_n\right)\right)\\
&=M\left((\k^n,\mathfrak{m}^\circ)\cap\left(v,\mathfrak{z}(\mathfrak{m}^{x_s})^{M\operatorname{-reg}}+x_n\right)\right)\\
&=M\left((\k^n,\mathfrak{m}^\circ)\cap\left(v,\mathfrak{z}(\mathfrak{m}^{x_s})^{\Gbreg}+x_n\right)\right)\\
&=(\k^n,\mathfrak{m}^\circ)\cap\left(M\left(v,\mathfrak{z}(\mathfrak{m}^{x_s})^{\Gbreg}+x_n\right)\right)\subseteq(\k^n,\mcirc)\cap\mathfrak{J}.
\end{align*}

(3) Let $(v_z,z_s+z_n)\in (\k^n,(\mathfrak{gl}_n^{y_s})^\circ)\cap \mathfrak{J}$. Then, $g \G^{z_s}g^{-1}=\G^{y_s}$  for some $g\in \G$. Now, $z_s\in (\mathfrak{gl}_n^{y_s})^\circ$ and $\G^{y_s}$ and $\G^{z_s}$ are connected. Therefore for dimension reasons $\G^{z_s}=\G^{y_s}$  and so $g\in N_{\G}(\G^{y_s})$. 

Then (2) for $\mathfrak m=\gl_n^{y_s}$ gives
\begin{align*}
(\k^n,  (\mathfrak{gl}_n^{y_s})^\circ)\cap\mathfrak{J}&\subseteq N_{\G}(\G^{y_s})\left(v,\mathfrak{z}(\mathfrak{gl}_n^{y_s})^{\G^{y_s}\operatorname{-reg}}+y_n\right)\\
&=(0,\mathfrak{z}(\mathfrak{gl}_n^{y_s})^{\G^{y_s}\operatorname{-reg}})+N_{\G}(\G^{y_s})(v,y_n).\end{align*}
The other inclusion follows from \eqref{eq:regM=regG} applied to $M=\G^{y_s}$. 
\end{proof}

\begin{proposition}\label{prop:enhanced-loc-closed}Enhanced Jordan classes are locally closed in $\Vb$ and smooth. 
\end{proposition}
\begin{proof}Let $\mathfrak{J}$ be the enhanced Jordan class of some $(v,x)=(0,x_s)+(v,x_n)$.   We first show that it is locally closed in $\Vb$.  Setting $M=G^{x_s}$ we have $x\in\mathfrak{m}^\circ$.  Let $U_{\mathfrak m}$, $U_{\mathfrak{gl}_n}$  and $\phi\colon \G\times^MU_{\mathfrak m}\to U_{\mathfrak{gl}_n}$ be the saturated principal open subsets of $\Vb_M$ and $\Vb$ and  the surjective \'etale morphism as in Lemma \ref{lem:slice-etale}.  Here $(v,x)\in U_{\mathfrak m}$ and so  $\mathfrak{J}\subset U_{\mathfrak{gl_n}}$.  It is enough to show that $\mathfrak{J}$ is locally closed  in $U_{\mathfrak{gl}_n}$. 
 This holds if and only if its preimage through $\phi$,  namely $\G\times^M\left(\mathfrak{J}\cap U_{\mathfrak m}\right)$ is locally closed in $\G\times^M U_{\mathfrak m}$, which is, in turn, equivalent to requiring that $\mathfrak{J}\cap U_{\mathfrak m}$ is locally closed in   $U_{\mathfrak m}$.
The latter is locally closed if and only if $(\k^n\mathfrak{m}^\circ)\cap\mathfrak{J}\cap U_{\mathfrak m}$ is locally closed in $(\k^n,\mathfrak{m}^\circ)\cap U_{\mathfrak m}$  if and only if $(\k^n,\mathfrak{m}^\circ)\cap\mathfrak{J}$ is locally closed in  $(\k^n,\mathfrak{m}^\circ)$. 
This last assertion follows from Lemma \ref{lem:intersection} (3). 
 Similarly, $\mathfrak{J}$ is smooth at $(v,x)$ if and only if $\G\times^M\left(\mathfrak{J}\cap U_{\mathfrak m}\right)$  is smooth at $1*(v,x)$ if and only if $\mathfrak{J}\cap U_{\mathfrak m}$ is smooth at $(v,x)$. 
 Intersecting once more with $(\k^n,\mathfrak{m}^\circ)$ we apply again Lemma   \ref{lem:intersection} (3), obtaining a disjoint union of smooth locally closed subsets.\end{proof}

\begin{proposition}\label{prop:dimension-enhanced} Let $\mathfrak{J}$ be the enhanced Jordan class of  $(v,x)=(0,x_s)+(v,x_n)$.   Then 
\begin{equation}\label{eq:dimension-enhanced}
\dim \mathfrak{J}=\dim \G(v,x)+\dim{\mathfrak z}(\gl_n^{x_s}).
\end{equation}
\end{proposition}
\begin{proof}We consider the dominant action morphism 
\begin{align*}\mu\colon \G\times \left(v,{\mathfrak z}(\gl_n^{x_s})^{\Gbreg}+x_n\right)&\to\overline{\mathfrak J}\\
\left(g,(v,y_s+x_n)\right)&\mapsto g(v,y_s+x_n).
\end{align*}
There is a non-empty open subset $U$ of $\overline{\mathfrak J}$ such that for any $u\in U$
\begin{align*}\dim{\mathfrak J}&=\dim\left(\G\times \left(v,{\mathfrak z}(\gl_n^{x_s})^{\Gbreg}+x_n\right)\right)-\dim \mu^{-1}(u)\\
&=\dim \G+\dim {\mathfrak z}(\gl_n^{x_s})-\dim \mu^{-1}(u).
\end{align*}
We estimate $\dim\mu^{-1}(u)$.  By irreducibility of $\overline{\mathfrak J}$,  density of $\mathfrak J$  and $\G$-equivariance of $\mu$ 
we may choose $u$ to lie in $\left(v,{\mathfrak z}(\gl_n^{x_s})^{\Gbreg}+x_n\right)$, so $u=(v,z+x_n)$ for some $z\in {\mathfrak z}(\gl_n^{x_s})^{\Gbreg}$. Let $S=\G(v,z+x_n)\cap \left(v,{\mathfrak z}(\gl_n^{x_s})^{\Gbreg}+x_n\right)$.  Then 
\[ N:=|S|\leq \left|\G z\cap {\mathfrak z}(\gl_n^{x_s})\right|<\infty\] because there are only finitely many $\G$-conjugates of a given element in a maximal toral subalgebra of $\gl_n$. 
Let then $S=\{s_1,\,\ldots,\, s_N\}$.  For $j=1,\,\ldots, N$, let $g_j\in \G$ be such that $s_j=g_j(v,z+x_n)$.  Then
\[\mu^{-1}(u)=\bigcup_{j=1}^N\left\{ (hg_j,s_j)\in \G\times\left(v,{\mathfrak z}(\gl_n^{x_s})^{\Gbreg}+x_n\right) ~|~h\in \G^{(v,z+x_n)}\right\}\]
so $\dim\mu^{-1}(u)=\dim  \G^{(v,z+x_n)}=\dim  \G^{(v,x)}$,  giving the statement. 
\end{proof}

\subsection{Geometry of exotic Jordan classes}
In this Section we turn our attention to the geometry of  exotic Jordan classes, showing that they are locally closed and smooth.

\begin{proposition}
\begin{enumerate}
\setlength{\itemsep}{4pt}
\item Exotic Jordan classes are locally closed.
\item Let $(v,x)=(0,x_s)+(v,x_n)\in\V$. Then
\begin{equation}
\label{eq:dimension-exotic}
\dim J(v,x)=\dim G(v,x)+\dim \V^{G^{x_s}}.\end{equation}
\end{enumerate}
\end{proposition}
\begin{proof}
(1) We consider the morphism $\psi$ from \eqref{eq:inclusions}. Recall that $\psi(\V)$ is closed in $\Vb_{2n}$ and that the restriction of $\psi$ to $1*\V$ is an isomorphism onto its image. Let $J$ be an exotic Jordan class.
Then $J$ is locally closed in $\V$ if and only if $\psi(J)$ is locally closed in $\psi(\V)$.  By Theorem \ref{thm:main-bijection} (3)  there is  an  enhanced Jordan class $\mathfrak J$ in $\Vb_{2n}$ such that $\psi(J)=\psi(\V)\cap\mathfrak J$. Proposition \ref{prop:enhanced-loc-closed} ensures that $\mathfrak J$ is locally closed  in $\Vb_{2n}$, whence $\psi(J)=\psi(\V)\cap\mathfrak J$ is locally closed in $\psi(\V)$.

(2) By  Theorem \ref{thm:main-bijection} (1)  we may assume that $(v,x)=\varphi(v',x')$ for some $(v',x')\in\Vb$. We consider the dominant morphism of affine varieties  $G\times^{\G}\overline{\mathfrak J(v',x')}\to \overline{J(v,x)}$ given by the restriction of $\varphi$, that we still denote by $\varphi$ for simplicity. 
By density of $J(v,x)$ in $\overline{J(v,x)}$ and equivariance of $\varphi$, there is a point $z$ in $J(v,x)\cap \varphi(\Vb)$ such that 
$\dim G\times^{\G}\overline{\mathfrak J(v',x')}-\dim\varphi^{-1}(z)=\dim\overline{J(v,x)}$. Replacing the representative of $J(v,x)$ in $\varphi(\Vb)$, we may assume $z=(v,x)$. 
Then 
\begin{align*}
\varphi^{-1}(v,x)&=\varphi^{-1}\varphi(v',x')=\{g*(w,y)\in G\times^{\G}\overline{\mathfrak J(v',x')}~|~gw=v',\,gy=x'\}\\
&=G^{(v,x)}\times^{\G^{(v',x')}}(v',x').
\end{align*}
Let $x_s$ and $x_s'$ denote respectively the semisimple parts of $x$ and $x'$. Invoking Proposition \ref{prop:dimension-enhanced} and Remark \ref{rem:reg-ss} we obtain
\begin{align*}
\dim\overline{J(v,x)}&=\dim G\times^{\G}\overline{\mathfrak J(v',x')}-\dim\varphi^{-1}(v,x)\\
&=\dim G-\dim\G+\dim\mathfrak z(\gl_n^{x'_s})+\dim \G(v',x')-\dim G^{(v,x)}+\dim \G^{(v',x')}\\
&=\dim G {(v,x)}+\dim\V^{G^{x_s}}.
\end{align*}
\end{proof}

\begin{proposition}
Exotic Jordan classes are smooth.
\end{proposition}
\begin{proof} Let $(v,x)\in \V$ with Jordan decomposition $(v,x)=(0,x_s)+(v,x_n)$. We may assume that $(v,x)\in\varphi(\Vb)$,   and that $x_s$ is as in  \eqref{eq:ss-exotic}. We identify elements in $\Vb$ and $\varphi(\Vb)$. Let 
$\mathbf M=\G^{x_s}\simeq\GL_\lb$, $\mathfrak m={\rm Lie}(\mathbf M)$, $M:=G^{x_s}\simeq\Sp_{2\lb}$, and $\mcirc$ be as in \eqref{eq:mcirc}. Let $\V_M$ be the exotic module for $M$ and $\Vb_{\mathbf M}$ the enhanced module for $\mathbf M$. 
It is not hard to verify that 
\begin{equation}\label{eq:emme}
\widetilde{X} := \{ (w,y_s+y_n) \in \V~|~G^{y_s} \subseteq {M} \}=\pi_{\V_M, M}^{-1}\pi_{\Vb_{\mathbf M},\mathbf M}(\k^n,\mathfrak m^\circ)= M(\k^{2n},\mathfrak m^\circ)
\end{equation} so $\widetilde{X}$ is an $\pi_{\V_M,M}$-saturated  open subset of $\V_M$ by Remark \ref{rem:saturation}.   
By Lemma \ref{lem:intersection} (3) and \eqref{eq:normalisers} we have 
\begin{align*}
\widetilde{X}\cap J&=M\left((\k^{2n},\mathfrak m^\circ)\cap \mathfrak J(v,x)\right)\\
&= M\left((\k^{2n},\mathfrak m^\circ)\cap \left((0,\mathfrak z(\mathfrak m))^{\mathbf M-\reg}+N_{\G}(\mathbf M)(v,x_n)\right)\right)\\
&\subseteq \widetilde{X}\cap \left((0,\V_M)^{M-\reg}+N_G(M)(v,x_n)\right)
\end{align*} and equality of these sets follows from \eqref{eq:emme}.
Hence, $\widetilde{X}\cap J$ is a disjoint union of locally closed smooth subsets of $\V_M$. Therefore $\widetilde{X}\cap J\cap U$ is smooth for all open neighbourhoods $U$ of $(v,x)$ in  $\V_M$ 
and so $G\times^{M}(\widetilde{X}\cap J\cap U)$ is smooth for all open neighbourhoods $U$ of $(v,x)$ in  $\V_M$.  

The statement follows if we show that there is an $M$-stable neighbourhood ${U}_0$ of $(v,x)$ in $\V_M$ such that  the restriction to  $G\times^{M}{U}_0$ of the
morphism $\chi\colon G\times^{M}\V_M\to\V$  of smooth varieties given by $(g*(w,z))\mapsto(gv,gz)$
is smooth. Indeed, if this is the case, $G({U}_0\cap\widetilde X)\cap J$ is a smooth neighbourhood of $(v,x)$ in $J$. 

We verify  the hypotheses of  \cite[Theorem 6.2]{BR}  for  $\chi$ at $(v,x)$. Closure of orbits and injectivity are readily proved as we did in Lemma \ref{lem:slice-etale}. 
Let ${\Vb}_{\mathbf H}$ be the enhanced module for  $\mathbf H:=\GL_{2n}^{\psi(0,x_s)}$  obtained using \eqref{eq:inclusions} for each component of $\lb$.
We define the  morphisms:
\begin{align*}
\widetilde{\chi}\colon G\times \V_M&\to\V,&&&\xi\colon \GL_{2n}\times{\Vb}_{\mathbf H}&\to\Vb_{2n},\\
g*(w,z)&\mapsto(gv,gz),&&&g'*(w',z')&\mapsto(g'w',g'z').
\end{align*}
To prove invertibility of $d\chi$  at $(1*(0,x_s))$ it is enough to show surjectivity of $d_{(1,0,x_s)}\widetilde{\chi}$. This differential is obtained by restriction to $\sp_{2n}\oplus \V_M$ of 
 $d_{(1,\psi(0,x_s))}\xi$. Using the calculation of the differential in the proof of Lemma \ref{lem:slice-etale}, we need to show that 
 for any $(w,y)\in \psi(\V)$ there are
 $A\in\sp_{2n}$ and $(u,z)\in\V_M$ such that 
\begin{align*}(w,y)=d_{(1,\psi(0,x_s))}\xi(A,(u,z))=(u,[A,\psi(0,x_s)]+\psi(0,z)), 
\end{align*}
i.e., with the identifications \eqref{eq:inclusions} and  \eqref{eq:ss-exotic}, 
that $[\sp_{2n},x_s]+\sum_{i=1}^r\bigwedge^2\k^{2\lambda_i}=\bigwedge^2\k^{2n}$. 
We recall that, as a $G$-module $\gl_{2n}=\sp_{2n}+\bigwedge^2\k^{2n}$. 
Now,  the subspace $[\sp_{2n},x_s]$ in the $G$-module $\gl_{2n}$  is the sum  of weight spaces  corresponding to weights $\chi$ satisfying $\chi(x_s)\neq0$. In other words, the non-zero entries of its elements lie outside of the blocks of the matrices in \eqref{eq:centraliser-exotic}, so 
$\sum_{i=1}^r\bigwedge^2\k^{2\lambda_i}\cap [\sp_{2n},x_s]=0$. 
Since $\dim [\sp_{2n},x_s]+\dim \sum_{i=1}^r\bigwedge^2\k^{2\lambda_i}=\dim\bigwedge^2\k^{2n}$ we have the statement. 
\end{proof}

\section{Induction of orbits and closures of classes: enhanced case}
\label{sec:inductionandclosures-enhanced}

Here we prove the results of Theorem~C in the enhanced setting.

\subsection{Induction of central orbits and closure of central enhanced Jordan classes}

We define induction of orbits from the enhanced module $\Vb_L$ of a Levi subgroup $L$ of $\GL_n$ to $\Vb$. We begin with the definition of induction of the orbit of an element
of the form $(v,x_s)$ where $x_s$ is semisimple and centralised by $L$, and $v\in\k^n$.

If $P$ is any parabolic subgroup containing $L$ as a Levi factor then we say that $P$ is \emph{adapted} to $L$ and $v$ if
$$\overline{P v} = \overline{L v}.$$

Let $F$ be a partial flag of $\k^n$, which we write $0 \subsetneq \k^{d_1} \subsetneq \k^{d_2} \subsetneq \cdots \subsetneq \k^{d_m} = \k^n$, and let $P$ be the parabolic subgroup of $\G$ stabilising $F$. A splitting of the flag $F$ gives rise to a direct sum decomposition $\k^n \cong \bigoplus_i \k^{d_i - d_{i-1}}$ and a Levi factor $L \subseteq P$ stabilising it. The decomposition also subtends $v = \sum_i v_i$.
\begin{lemma}\label{lem:good-parabolic}
The parabolic $P$ is adapted to $L$ and $v$ if and only if $\{ 1\le i \le m \mid v_i \ne 0\} = \{1,2,...,k\}$ for some $1\le k \le m$. In particular, adapted parabolics exist.
\end{lemma}
The proof is quite easy, so we omit it.

\medskip

Let $(v,x_s)\in \Vb$ and $L$ be a Levi subgroup of $\G$ such that $L\subseteq\G^{x_s}$.  Choose a  parabolic subgroup $P$ of $\G$ adapted to $L$ and $v$, with Levi decomposition $P=LU_P$, and let $\mathfrak{u}_P={\rm Lie}(U_P)$.  By construction 
\[\overline{L(v,x_s)}+(0,\mathfrak{u}_P)=(\overline{Lv},x_s+\mathfrak{u}_P)=(\overline{Pv},x_s+\mathfrak{u}_P)\] is $P$-stable.
 
We consider the associated bundle $\G\times^P\left(\overline{Lv},x_s+\mathfrak{u}_P\right)$ and the $\G$-equivariant morphism
\begin{align}\label{eq:associated-bundle}
\widetilde{\pi}\colon\G\times^P\left(\overline{Lv},x_s+\mathfrak{u}_P\right)&\to\Vb\\
\nonumber g*(w,y)&\mapsto (gw,gy) 
\end{align}

It is a slight generalisation of the resolution of singularities in \cite[\textsection 3]{ah}.

\begin{lemma}\label{lem:induction0-step1}Let $(v,x_s)\in \Vb$,  let $L$ be a Levi subgroup of $\G$ with  $L\subset {\G}^{x_s}$ and  ${\mathfrak l}={\rm Lie}(L)$,  let $P$ be a parabolic adapted to $L$ and $v$ with Levi decomposition $P=LU_P$ and ${\mathfrak u}_P={\rm Lie}(U_P)$.
\begin{enumerate}
\setlength{\itemsep}{4pt}
\item The morphism \eqref{eq:associated-bundle} is proper and therefore closed.  
\item Its image is the closure of  a $\G$-orbit. We denote it by ${\bO}$.
\item The orbit ${\bO}$ has  a representative in $(v,x_s+(\mathfrak{u}_P\cap (\gl_n^{x_s})^{\Gbreg}))$. 
\item If  $x_s\in{\mf z}(\mf l)^{\Gbreg}$,  then $ {\bO}=\G(v,x_s)$.
\item For $p$ as in \eqref{eq:projections} there holds $p({\bO})={\rm Ind}_{\mathfrak l,x_s}^{\mathfrak{gl}_n}(\{0\})$, in particular, if  $x_s=0$, then $p({\bO})$ is  the Richardson orbit in $\gl_n$ corresponding to $L$. 
\item The morphism $\widetilde{\pi}$ is birational. 
\end{enumerate}
\end{lemma}
\begin{proof} 
\begin{enumerate}
\item The morphism $\widetilde{\pi}$  is closed because it corresponds to a projection onto $\G(\overline{Lv},x_s+\mathfrak{u}_P)$ after the identification of $\G\times^P (\overline{Lv},x_s+\mathfrak{u}_P)$ with the incidence variety
\[ \{(gP,z)\in \G/P\times \G(\overline{Lv},x_s+\mathfrak{u}_P)~| ~g^{-1}z\in(\overline{Lv},x_s+\mathfrak{u}_P)\}.\]

\item The image of $\widetilde{\pi}$  is $\G$-stable because $\widetilde{\pi}$  is $\G$-equivariant,  irreducible because $\left(\overline{Lv},x_s+\mathfrak{u}_P\right)$ is irreducible, and closed by (1). Therefore it is enough to show that  ${\rm Im}\,\widetilde{\pi}$ contains finitely-many $\G$-orbits.   
By looking at the component in $\gl_n$, we see as in \cite[Section 2.1]{bo-inv} that the elements in ${\rm Im}\,\widetilde{\pi}$ have semisimple part contained in the $\G$-orbit of $x_s$.  Hence, its image is contained in
$\G(\k^n ,x_s+\gl_n^{x_s}\cap\mathcal N(\gl_n))$.  Now,  $(\k^n, x_s+\gl_n^{x_s}\cap\mathcal N(\gl_n))$ is obtained by shifting by $(0,x_s)$ the enhanced nilcone for $\G^{x_s}$,  so it consists of finitely-many $\G^{x_s}$-orbits.

\item Note that 
\[{\bO}=\left(\G(\overline{Lv},x_s+\mathfrak{u}_P)\right)^{\Gbreg}=\G(\overline{Lv},x_s+\mathfrak{u}_P)^{\Gbreg}.\]
Let $\dim {\bO}=d$,  that is, ${\bO}$ lies in the level set $\Vb_{(d)}$ consisting of orbits of dimension $d$. 
As in  (2) we see that $x_s+\mathfrak{u}_P$ is $P$-conjugate to an element in $x_s+(\mathfrak{u}_P\cap \gl_n^{x_s})$.
Therefore, ${\bO}$ has a representative in $(\overline{Pv},x_s+(\mathfrak{u}_P\cap \gl_n^{x_s}))\cap\Vb_{(d)}$. As $P$ is adapted to $L$ and $v$, 
\[{\bO}\subseteq \G (\overline{Lv},x_s+(\mathfrak{u}_P\cap \gl_n^{x_s})\cap \Vb_{(d)}\subseteq  \G(\overline{Lv},x_s+\mathfrak{u}_P)\cap \Vb_{(d)}={\bO}\]
and the inclusions are equalities. 

Now,  the subsets $(\overline{Lv},x_s+(\mathfrak{u}_P\cap \gl_n^{x_s}))\cap \Vb_{(d)}$ and  $({Lv},x_s+(\mathfrak{u}_P\cap\gl_n^{x_s})^{\Gbreg})$ are open, respectively dense
 in the irreducible subset $(\overline{Lv},x_s+(\mathfrak{u}_P\cap \gl_n^{x_s}))$, so 
\[\emptyset\neq(\overline{Lv},x_s+(\mathfrak{u}_P\cap \gl_n^{x_s}))\cap \Vb_{(d)}\cap ({Lv},x_s+(\mathfrak{u}_P\cap\gl_n^{x_s})^{\Gbreg})\subseteq {\bO}.\]

Acting by $L$, we obtain   a representative in $(v,x_s+(\mathfrak{u}_P\cap\gl_n^{x_s})^{\Gbreg})$. 
\item If $x_s\in{\mf z}(\mf l)^{\Gbreg}$,  then $\gl_n^{x_s}=\mf l$, whence $\mathfrak{u}_P\cap \gl_n^{x_s}=0$, and so 
$(v,x_s)\in {\bO}$. 
\item Follows  because ${\rm Ind}_{\mathfrak l,x_s}^{\mathfrak{gl}_n}(\{0\})=\G (x_s+\mathfrak{u}_P)^{\Gbreg}$, by the arguments in \cite[\S 2.1]{bo-inv}, \cite{LS}. 

\item  Choose a representative $\widetilde{x}=(v,x)=(v,x_s+x_n)$ of ${\bO}$  as in (3)  and consider the commutative diagram of $\G$-equivariant morphisms
\begin{equation*}
\begin{tikzcd}
{\G} \times^P (\overline{L v},x_s+\mathfrak{u}_P) \arrow[d, swap,"\id \times p"] \arrow[rr, "\widetilde{\pi}"]&& \overline{{\bO}}\subseteq \Vb \arrow[d, "p"] \\
{\G} \times^P( x_s+\mathfrak{u}_P) \arrow[rr,"\pi"] && \overline{{\rm Ind}_{\mathfrak{l},x_s}^{\mathfrak{gl}_n} (\{0\}) } \subseteq \mathfrak{gl}_n
\end{tikzcd} 
\end{equation*}
The characteristic-free criterion in the proof of \cite[Lemma 7.10]{BorhoKraft} asserts that  \begin{equation*}|\widetilde{\pi}^{-1}(v,x_s+x_n)|=[\G^{(v,x_s+x_n)}:P^{(v,x_s+x_n)}]\left|({{\bO}}\cap (\overline{L v},x_s+\mathfrak{u}_P))/P\right|.\end{equation*} It follows from (5) that $G^{x_s}x_n$ is the Richardson orbit for $G^{x_s}$ induced using the parabolic subgroup $P^{x_s}$, so  $[\G^{x_s+x_n}:P^{x_s+x_n}]\left|p({{\bO}})\cap (x_s+\mathfrak{u}_P\cap \gl_n^{x_s})/P^{x_s}\right|=1$. 
Therefore 
\begin{equation}\label{eq:index}
[\G^{(v,x_s+x_n)}:P^{(v,x_s+x_n)}]=[\G^v\cap\G^{x_s+x_n}:P^{(v,x_s+x_n)}]=[\G^v\cap P^{x_s+x_n}:P^{(v,x_s+x_n)}]=1.
\end{equation} 
Let now $(w,x_s+u)\in {{\bO}}\cap (\overline{L v},x_s+\mathfrak{u}_P)$. Since $\left|p({{\bO}})\cap (x_s+\mathfrak{u}_P\cap\gl_n^{x_s})/P^{x_s}\right|=1$, there is $w'\in \overline{L v}$ such that 
$(w',x_s+x_n)\in P(w,x_s+u)$. Since $P(w,x_s+u)\subset{{\bO}}$, there is  $g\in\G$ such that $g(w,x_s+u)=(v,x_s+x_n)$. But then, $g\in \G^{x_s+x_n}\subset P^{x_s+x_n}$, whence   $P(v,x_s+x_n)=P(w,x_s+u)$. Combining with \eqref{eq:index} we have $|\widetilde{\pi}^{-1}(v,x_s+x_n)|=1$.
\end{enumerate}
\end{proof}
\medskip

We denote the orbit ${\bO}$ constructed as in Lemma \ref{lem:induction0-step1} by ${\rm Ind}_{\Vb_L,P}^{\Vb}L(v,x_s)$ and we call it the \emph{induced enhanced orbit} of $L(v,x_s)$. We observe that induced orbits have dimension properties generalising those of induced adjoint orbits.

\begin{lemma}\label{lem:induction0-step2}Let $x_s$, $L$, $\mathfrak l$, $v$ and $P$ be as Lemma \ref{lem:induction0-step1}.  Then
\begin{equation}\label{eq:estimate-codim}\codim_{\Vb}{\rm Ind}_{\Vb_L,P}^{\Vb}L(v,x_s)=\codim_{\Vb_L}L(v,x_s).\end{equation}
\end{lemma}
\begin{proof}
Let $P=LU_P$ be the Levi decomposition of $P$ and  $\mathfrak{u}_P={\rm Lie}(U_P)$. By the fibre theorem there is a dense subset $U$ of $\overline{{\rm Ind}_{\Vb_L,P}^{\Vb}L(v,x_s)}$  such that for any $u\in U$  we have
\begin{align*}\codim_{\Vb}({\rm Ind}_{\Vb_L,P}^{\Vb}L(v,x_s))&=\dim\G+n-\dim\left(\G\times^P\left( \overline{Lv},x_s+\mathfrak{u}_P\right)\right)+\dim\widetilde{\pi}^{-1}(u)\\
&=n+\dim P-\dim U_P-\dim Lv+\dim\widetilde{\pi}^{-1}(u)\\
&=n+\dim L-\dim Lv+\dim\widetilde{\pi}^{-1}(u)\\
&=\codim_{\Vb_L}L(v,x_s)+\dim\widetilde{\pi}^{-1}(u).
\end{align*}
The conclusion follows from Lemma \ref{lem:induction0-step1} (6).
\end{proof}

\subsection{The closure of the enhanced class of $(v,x_s)$}
We are now in a position to describe the closure of enhanced Jordan classes $\mathfrak J$ such that $p(\mathfrak J)$ consists of semisimple elements. 

\begin{proposition}\label{prop:unique}
Let $(v,x_s)\in\Vb$, with $x_s\in\gl_n$ semisimple, let $L={\G}^{x_s}$ and $\mathfrak{l}=\gl_n^{x_s}$, and let $P$ be a parabolic subgroup of $\G$ adapted to $L$ and $v$. 
Let ${\mathfrak u}_P$ be the nilpotent radical of ${\rm Lie}(P)$.  Then
\begin{equation}\label{eq:closure-enhanced}
\overline{{\mathfrak J}(v,x_s)}=\G\left(\overline{Lv},\mathfrak{z}(\mathfrak l)+\mathfrak{u}_P\right),\quad\mbox{ and }\quad \overline{{\mathfrak J}(v,x_s)}\cap{\mathcal N}( \Vb )=\G\left(\overline{Lv},\mathfrak{u}_P\right)=\overline{{\rm Ind}_{\Vb_L,P}^{\Vb}L(v,0)}.\end{equation}
In particular, $ \overline{{\mathfrak J}(v,x_s)}\cap{\mathcal N}(\Vb)$ is irreducible and ${\rm Ind}_{\Vb_L,P}^{\Vb}L(v,0)$ is independent of the choice of the  adapted parabolic. 
\end{proposition}
\begin{proof}
Since $P$ is adapted to $L$ and $v$, the closed and irreducible subset $\left(\overline{Lv},\mathfrak{z}(\mathfrak l)+\mathfrak{u}_P\right)$ is $P$-stable. We consider the 
associated bundle  $\G\times^P\left(\overline{Lv},\mathfrak{z}(\mathfrak l)+\mathfrak{u}_P\right)$ and the $\G$-equivariant morphism
\begin{align*}
\widetilde{\mu}\colon \G\times^P\left(\overline{Lv},\mathfrak{z}(\mathfrak l)+\mathfrak{u}_P\right)&\to\Vb\\
g *(w,y)&\mapsto (gw, gy).
\end{align*}
Arguing as we did for Lemma \ref{lem:induction0-step1} (1) we see that $\widetilde{\mu}$ is closed hence its image $\G\left(\overline{Lv},\mathfrak{z}(\mathfrak l)+\mathfrak{u}_P\right)$ is closed and irreducible.
Now, ${\mathfrak J}(v,x_s)=\G\left(v,\mathfrak{z}(\mathfrak l)^{\Gbreg}\right)\subseteq \G\left(\overline{Lv},\mathfrak{z}(\mathfrak l)+\mathfrak{u}_P\right)$ whence
$\overline{{\mathfrak J}(v,x_s)}\subseteq \G\left(\overline{Lv},\mathfrak{z}(\mathfrak l)+\mathfrak{u}_P\right)$. 
Moreover, by \eqref{eq:dimension-enhanced}
\begin{align*}
\dim  {\mathfrak J}(v,x_s)&=\dim\G (v,x_s)+\dim\mathfrak{z}(\mathfrak l)\\
&=\dim\G-\dim \G^{x_s}+\dim \G^{x_s}v+\dim\mathfrak{z}(\mathfrak l)\\
&=\dim\G-\dim P+\dim\mathfrak{u}_P+\dim \overline{Lv}+\dim\mathfrak{z}(\mathfrak l)\\
&= \dim \G\times^P\left(\overline{Lv},\mathfrak{z}(\mathfrak l)+\mathfrak{u}_P\right)\geq \dim \G\left(\overline{Lv},\mathfrak{z}(\mathfrak l)+\mathfrak{u}_P\right)\geq \dim  {\mathfrak J}(v,x).
\end{align*}
Hence $\dim  \overline{{\mathfrak J}(v,x)}=\dim \G\left(\overline{Lv},\mathfrak{z}(\mathfrak l)+\mathfrak{u}_P,\right)$ and the first assertion follows from irreducibility of  $ \G\left(\overline{Lv},\mathfrak{z}(\mathfrak l)+\mathfrak{u}_P,\right)$.
 The  second equality follows because if $(w,y)=(w,z+x_n)\in \left(\overline{Lv}, \mathfrak{z}(\mathfrak l)+\mathfrak u_P\right)$, then it lies in 
 ${\mathcal N}(\Vb)$ if and only if $z=0$.
\end{proof}

In light of Proposition \ref{prop:unique} we will drop the subscript $P$ in ${\rm Ind}_{\Vb_L,P}^{\Vb}L(v,0)$ writing simply ${\rm Ind}_{\Vb_L}^{\Vb}L(v,0)$. The induced orbits from Lemma \ref{lem:induction0-step1} can all be reduced to the case in which $x_s=0$ as we now see.

\begin{proposition}\label{prop:centraliser}
 Let $L$ be a Levi subgroup of a parabolic subgroup of $\G$  and let ${\mathfrak l}={\rm Lie}(L)$.  Let  $x_s\in {\mathfrak z}({\mathfrak l})$ and $M:=\G^{x_s}$, and let $\Vb_M$ and $\Vb_L$ be respectively the enhanced module of $L$ and $M$. Let $v\in \k^n$ and $P$ be a parabolic subgroup adapted to $L$ and $v$. Then  ${\rm Ind}_{\Vb_L,P}^{\Vb}L(v,x_s)$ is independent of the adapted parabolic $P$ and 
\begin{equation}
{\rm Ind}_{\Vb_L,P}^{\Vb}L(v,x_s)=\G\left((0,x_s)+{\rm Ind}_{\Vb_L}^{\Vb_M}L(v,0)\right)
\end{equation}
that is,  the semisimple part of ${\rm Ind}_{\Vb_L,P}^{\Vb}L(v,x_s)$ is $x_s$ and the nilpotent part can be taken in the induced orbit ${\rm Ind}_{\Vb_L}^{\Vb_M}L(v,0)$ for the enhanced module corresponding to $M$.   
\end{proposition}
\begin{proof}
Observe that  $\mathfrak l$  is a Levi subalgebra of a parabolic subalgebra of ${\mathfrak m}:={\rm Lie}(M)=\gl_n^{x_s}$ and  
and that $L$ is a Levi subgroup of the  parabolic subgroup $P^{x_s}$  of $M$. In addition,  $P^{x_s}v\subseteq Pv\subseteq\overline{Lv}$, that is, 
 $P^{x_s}$ is adapted to $L$ and $v$ for the enhanced module $\Vb_M$.  Moreover, ${\mathfrak u}_P\cap\mathfrak m$ is irreducible, as it is the unipotent radical of ${\rm Lie}(P^{x_s})$.  By Lemma \ref{lem:induction0-step1} (3)
\begin{align*}
{\rm Ind}_{\Vb_L,P}^{\Vb}L(v,x_s)&=\G\left(v,x_s+(\mathfrak{u}_P\cap \mm)^{\Gbreg}\right).
\end{align*}
A density argument shows that  $x_s+(\mathfrak{u}_P\cap \mm)^{M\operatorname{-reg}}$ meets the regular locus of $x_s+(\mathfrak{u}_P\cap \mm)$,  which in turn, is given by elements of the form $x_s+u$ for $u$ in the $M$-regular part of $\mathfrak{u}_P\cap \mm$ by the properties of the Jordan decomposition. Therefore
\begin{align*}
{\rm Ind}_{\Vb_L,P}^{\Vb}L(v,x_s)&=\G\left((0,x_s)+M\left((\mathfrak{u}_P\cap \mm)^{M\operatorname{-reg}},v\right)\right)\\&=\G\left((0,x_s)+{\rm Ind}_{\Vb_L}^{\Vb_M}L(v,0)\right).
\end{align*}
\end{proof}

Thanks to Proposition \ref{prop:centraliser} we may drop the subscript $P$ in the expression ${\rm Ind}_{\Vb_L,P}^{\Vb}L(v,x_s)$ for any $x_s\in{\mf z}(\mf l)$.

\begin{proposition}Let $(v,x_s)\in\Vb$, with $x_s$ semisimple,  $L=\G^{x_s}$,  and $\mathfrak{l}=\gl_n^{x_s}$.  Then
\begin{equation}\label{eq:closure-ss-Jc}
\overline{{\mathfrak J}(v,x_s)}=\bigcup_{z\in{\mathfrak z}(\mathfrak l)}\overline{{\rm Ind}_{\Vb_L}^{\Vb}(v,z)}=\bigcup_{z\in{\mathfrak z}(\mathfrak l)}\overline{\G\left((0,z)+{\rm Ind}_{\Vb_L}^{\Vb_{\G^z}}L(v,0)\right)}.
\end{equation}
and 
\begin{equation}\label{eq:reg-closure-ss-Jc}
\overline{{\mathfrak J}(v,x_s)}^{\Gbreg}=
\bigcup_{z \in {\mathfrak z}(\mathfrak l)} {\rm Ind}_{\Vb_L}^{\Vb}(v,z) = 
\bigcup_{z \in {\mathfrak z}(\mathfrak l)} \G  \left( (0,z) + {\rm Ind}_{\Vb_L}^{\Vb_{\G^z}}L(v,0)\right).
\end{equation}
\end{proposition}
\begin{proof}
Let $P=LU_P$ be a parabolic subgroup adapted to $L$ and $v$ and let ${\mathfrak u}_P={\rm Lie}(U_P)$.
Then \eqref{eq:closure-enhanced} gives
\begin{align*}
\overline{{\mathfrak J}(v,x_s)}&=\G(\bigcup_{z\in{\mathfrak z}(\mathfrak l)}(\overline{L v},z+{\mathfrak u}_P))=\bigcup_{z\in{\mathfrak z}(\mathfrak l)}\G(\overline{Lv},z+{\mathfrak u}_P)
=\bigcup_{z\in{\mathfrak z}(\mathfrak l)}\overline{{\rm Ind}_{\Vb_L}^{\Vb}L(v,z)}
\end{align*} so \eqref{eq:closure-ss-Jc} follows from Proposition  \ref{prop:centraliser}.   Then, \eqref{eq:reg-closure-ss-Jc}  follows from \eqref{eq:estimate-codim}.
\end{proof}

\subsection{The combinatorics of inducing from a vector} 

We now determine the bipartition associated to the enhanced orbit induced from a nilpotent element $(v,0) \in \cN(\Vb_L)$ for a chosen Levi subgroup $L$. 

We start by choosing a composition $(n_1,...,n_r)$ of $n$, meaning $\sum_j n_j = n$, and a vector space decomposition $\k^n = \bigoplus_{j=1}^r \k^{n_j}$. The subgroup of $\G$ stabilising the subspaces $\k^{n_j}$ is a Levi subgroup $L \subseteq \G$. We have $L \simeq\prod_{j=1}^r\GL_{n_j}$. 

For later use in this section we write $v_{1,j},...,v_{n_j, j}$ for the standard basis in $\k^{n_j}$, so that $\{v_{i,j} \, : 1\le j \le r, \ 1\le i \le n_i\}$ is a basis for $\k^n$. 

The enhanced nilpotent cone $\mathcal N(\Vb_L)$ for $L$ has the form $\mathcal N(\Vb_L) = \prod_{i=1}^{r} \k^{n_{i}} \times \mathcal{N}(\mathfrak{gl}_{n_{i}})$. 
As we recalled in Section~\ref{ss:nilpandss} the $\G$-orbits in $\cN(\Vb)$ are parameterised by bipartitions $\mathcal{Q}_n$, therefore $L$-orbits in $\cN(\Vb_L)$ are parameterised by elements of $\prod_{i=1}^r \mathcal{Q}_{n_i}$. We note that $\G(v, 0) \in \cN(\Vb)$ corresponds to $(1^n; \emptyset)$ if $v \ne 0$ whilst it corresponds to $(\emptyset; 1^n)$ if $v = 0$.

Pick an element $v\in \k^n$ and write $v = v^1 + \cdots + v^r$ for the decomposition with $v^i \in \k^{n_i}$. After reordering $(n_1,...,n_r)$ we can assume (without loss of generality) that $v^i \ne 0$ for $i=1,...,k$ and $v^{k+1} = \cdots = v^r = 0$, for some fixed $k$. In fact we can (and shall) assume that $v^j = \sum_{i=1}^{n_j} v_{i,j}$.

Order the basis $\{v_{i,j}\}$ so that $v_{i,j}$ precedes $v_{i', j'}$ whenever $j \le j'$. Let $P$ be the parabolic subgroup containing $L$ as a Levi factor, which also contains the upper-triangular Borel subalgebra. Then by Lemma~\ref{lem:good-parabolic} and by our choice of ordering of $\{v_{i,j}\}$ we see that $P$ is adapted to $L$ and $v$. Combining the above remarks, the orbit $Lv$ can be written in bipartition notation, as follows
$$Lv = \prod_{i=1}^{k} \bO_{(1^{n_i},\emptyset)} \times \prod_{i=k+1}^{r}  \bO_{(\emptyset, 1^{n_i})}.$$

In Lemma~\ref{lem:induction0-step1} we showed that the data $(P, L, Lv)$ give rise to a morphism $\G \times^P (\overline{Lv}, {\mf u}_P) \to \Vb$ and that there is an orbit $\bO:={\rm Ind}_{\Vb_L}^{\Vb}L(v,0)$ which is dense in the image.

\begin{proposition}
\label{prop:combinatorics-inducedfromzero}
The bipartition corresponding to $\bO$ is $(\mu ; \nu)$, where $\mu_{j}= |\{ i \in \{1,\ldots,k\} \; : n_i \geq j \} |$ and $\nu_{j}= |\{ i \in \{k+1,\ldots, r\} \; : n_i \geq j \}|$. 
\end{proposition}

\begin{proof}
We choose a nilpotent element $x_n \in \Ind_{\l,0}^\g(\{0\}) \cap {\mf u}_P$, and fix $v$ as above. We claim that the bipartition of $\G(v,x_n)$ is precisely the one described in the statement of the current proposition. In order to justify this we will make a good choice of $x_n$ so that $\{v_{i,j}\}_{i,j}$ are a normal basis for $(v, x_n)$ (possibly after relabelling basis).

Define
\begin{eqnarray}
\label{eq:xndefined}
x_n v_{i,j} := \left\{ \begin{array}{rl} v_{i,l} & \text{ if } n_l \ge n_j \text{ and } n_q < n_j \text{ for } l < q < j \\ 0 & \text{ if no such } l \text{ exists}; \end{array} \right.
\end{eqnarray}
To aid the reader in understanding the definition of $x_n$ we provide a diagram: let $(4,2,3,5)$ be a composition of $14$. We then construct a single-column tableaux placing the vector $v_{i,j}$ in row $i$ and column $j$ with the nilpotent $x_n$ acting as follows:
\begin{center}
\begin{tikzpicture}[scale=0.5]
\draw (0,0) -- (0,4) -- (1,4) -- (1,0) -- cycle;
\draw (0,1) -- (1,1);
\draw (0,2) -- (1,2); 
\draw (0,3) -- (1,3);

\node at (0.5,3.5) {\tiny $v_{11}$};
\node at (0.5,2.5) {\tiny $v_{21}$};
\node at (0.5,1.5) {\tiny $v_{31}$};
\node at (0.5,0.5) {\tiny $v_{41}$};

\draw (3,2) -- (3,4) -- (4,4) -- (4,2) -- cycle;
\draw (3,3) -- (4,3); 

\node at (3.5,3.5) {\tiny $v_{12}$};
\node at (3.5,2.5) {\tiny $v_{22}$};

\draw (6,1) -- (6,4) -- (7,4) -- (7,1) -- cycle;
\draw (6,2) -- (7,2); 
\draw (6,3) -- (7,3);

\node at (6.5,3.5) {\tiny $v_{13}$};
\node at (6.5,2.5) {\tiny $v_{23}$};
\node at (6.5,1.5) {\tiny $v_{33}$};

\draw (9,-1) -- (9,4) -- (10,4) -- (10,-1) -- cycle;
\draw (9,0) -- (10,0);
\draw (9,1) -- (10,1);
\draw (9,2) -- (10,2); 
\draw (9,3) -- (10,3);

\node at (9.5,3.5) {\tiny $v_{14}$};
\node at (9.5,2.5) {\tiny $v_{24}$};
\node at (9.5,1.5) {\tiny $v_{34}$};
\node at (9.5,0.5) {\tiny $v_{44}$};
\node at (9.5,-0.5) {\tiny $v_{54}$};


\draw[->] (-0.5,3.5) to node[above] {\tiny $x_n$} (-1.5,3.5)node[left]{\tiny $0$};
\draw[->] (-0.5,2.5) to (-1.5,2.5)node[left]{\tiny $0$};
\draw[->] (-0.5,1.5) to (-1.5,1.5)node[left]{\tiny $0$};
\draw[->] (-0.5,0.5) to (-1.5,0.5)node[left]{\tiny $0$};

\draw[->] (2.5,3.5) to node[above] {\tiny $x_n$} (1.5,3.5);
\draw[->] (2.5,2.5) to (1.5,2.5);

\draw[->] (5.5,3.5) to node[above] {\tiny $x_n$} (4.5,3.5);
\draw[->] (5.5,2.5) to (4.5,2.5);
\draw[->] (5.5,1.5) to (1.5,1.5);

\draw[->] (8.5,3.5) to node[above] {\tiny $x_n$} (7.5,3.5);
\draw[->] (8.5,2.5) to (7.5,2.5);
\draw[->] (8.5,1.5) to (7.5,1.5);
\draw[->] (8.5,0.5) to (1.5,0.5);
\draw[->] (8.5,-0.5) to (-1.5,-0.5)node[left]{\tiny $0$};

\end{tikzpicture}
\end{center}

Let $\lambda \vdash n$ denote the partition obtained by reordering $(n_1,...,n_r)$ to a partition, and taking the transpose. Equivalently, $\lambda = (\lambda_1,...,\lambda_{m})$ where $\lambda_i := |\{ 1\le j \le r \, : n_j \ge i\}|$ and $m = \max \{n_1,...,n_r\}$. Note that $\lambda_i = \mu_i + \nu_i$, with $(\mu; \nu)$ defined in accordance with the statement of the proposition.

For illustrative purposes, we suppose that $w = \sum_{j=1}^2 \sum_{i=1}^{n_i} v_{i,j}$, and we observe that, in our example, $x_n |_{\k[x_n]w}$ has Jordan block sizes $(2,2,1,1)$.

It follows directly from \eqref{eq:xndefined} that $x_n$ has Jordan block sizes given by $\lambda$. Furthermore it is not hard to see that $x_n|_{\k[x_n] v}$ has Jordan block sizes given by $\mu$. We deduce that $\G(v, x_n) = \bO_{(\mu; \nu)} \subseteq \cN(\Vb)$, using \cite[Corollary~2.9]{ah}.

Now   $(v, x_n)\in (v,\mathfrak u_P)$, so it  lies in the image of the morphism \eqref{eq:associated-bundle}, that is  $\G(v,x_n)\subset\overline{\bO} \subseteq \Vb$.

We we will show that $\overline{\G(v,x_n)}=\overline\bO$ by counting the dimensions of both varieties, and noting that they are equal.

First of all Lemma~\ref{lem:induction0-step1}(6), \cite{LS} and part  (8) of \cite[Proposition~2.8]{ah} imply that
\begin{eqnarray}
\label{eq:someeq1}
\begin{array}{rcl}
\dim \bO & = & \dim \G \times^P (\overline{Lv}, {\mf u}_P) \\ 
&=&\dim\G-\dim L+\dim Lv \\ &= & n^2 - \sum_{i=1}^r n_i^2 + \sum_{i=1}^k n_i \\ &=& \dim \G x_n + \sum_i \mu_i \\
&=& \dim \G(v, x_n).
\end{array}
\end{eqnarray}
Therefore $\bO=\G(v,x_n)=\bO_{(\mu;\nu)}$.
\end{proof}

\begin{corollary}\label{cor:all-induced}
\begin{enumerate}
\item[(1)] Every orbit in ${\mathcal N}(\Vb)$ is equal to $\Ind_{\Vb_L}^\Vb (Lv)$ for some choice of Levi subgroup $L$ and orbit $Lv \subseteq \k^n$.
\item[(2)] The pair $(L, Lv)$ in (1) is unique upto $\G$-conjugacy.
\item[(3)] The rigid orbits in $\cN(\Vb)$ are precisely those of the form $\G(v,0)$ where $v\in \k^n$. These are the orbits with bipartition $(1^n; \emptyset)$ or $(\emptyset; 1^n)$.
\end{enumerate}
\end{corollary}
\begin{proof}
(1) Let $(\mu; \nu) \in \mathcal{Q}_n$ be a choice of bipartition. We let $(n_1,...,n_k)$ be the lengths of the columns of $\mu$. Explicitly, this is the transpose partition to $\mu$, ordered so that $n_1 \le \cdots \le n_k$, where $k = \mu_1$. We also let $n_{k+1},\dots,n_r$ denote the lengths of the columns of a Young diagram representing $\nu$. Then the data $(n_1,\dots,n_r)$, together with $k$, gives rise to a choice of Levi subgroup $L$ and a vector $v \in \k^n$, as we explained prior to Proposition~\ref{prop:combinatorics-inducedfromzero}. Furthermore {\it loc. cit.}  shows that $\Ind_{\Vb_L}^\Vb(Lv) = \bO_{(\mu;\nu)}$, which proves (1).

(2) We claim that $\G$-orbits of pairs $(L, Lv)$ are in 1-1 correspondence with bipartitions. 

Let $(L, Lv)$ be a pair of Levi $L$ and $L$-orbit in $\k^n$. The $\G$-orbit of $(L, Lv)$ is uniquely determined by the $\G$-orbit of $L$ and the $N_\G(L)$-orbit of $Lv$. The $\G$-orbit of $L$ is classified by the equivalence class of the composition $(n_1,\dots ,n_r)$ where $L \simeq \prod_{i=1}^r \GL_{n_i}$, and equivalence of compositions is given by reordering. Clearly every equivalence class is represented by a unique partition.

With $L \simeq \prod_{i=1}^r \GL_{n_i}$ fixed, we can define a decomposition $\k^n = \bigoplus_{i=1}^r \k^{n_i}$ where $\k^{n_i} \subseteq \k^n$ is the subspace invariant under $\prod_{j\ne i} \GL_{n_j}\subseteq L$. Now $v \in \k^n$ decomposes as $v = \sum_{i=1}^r v^i$ and we refer to the set $\{1\le  i \le r \, : v^i \ne 0\}$ as the support of $v$. Note that all elements of $Lv$ have the same support. Then the $N_\G(L)$-orbit of $Lv$ is determined by the support of $v$, up to permuting the indexes $1,...,r$ for which the parts of the composition are equal.

Now from $(L, Lv)$ we form a bipartition $(\mu; \nu) \in \mathcal{Q}_n$ by letting $\mu$ be the ordered list consisting of all $n_i$ such that $i$ lies in the support of $v$, and $\nu$ to be the ordered list of all $n_i$ such that $i$ does not lie in the support. The fact that this gives a bijection is not hard to see.

Now according to part (1) there is a surjective map from the $\G$-conjugacy classes of pairs $(L, Lv)$ to the $\G$-orbits in $\cN(\Vb)$. Since both sets are parameterised by $\mathcal{Q}_n$ this map is a bijection.

(3) By part (1) an orbit is not rigid if the bipartition is neither $(1^n ; \emptyset)$ and $(\emptyset; 1^n)$. On the other hand, $(1^n; \emptyset)$  and $(\emptyset; 1^n)$ are precisely the $\G$-orbits in $\cN(\Vb)$ which admit representatives of the form $(v,x)$ with $x = 0$. Neither of these can be induced from a proper Levi subgroup, thanks to Lemma~\ref{lem:induction0-step1}(5).
\end{proof}

\begin{example}
Let $n=13$, so that $\mathcal{N}(\Vb)=\k^{13} \times \mathcal{N}(\mathfrak{gl}_{13})$, and suppose we wish to induce the nilpotent orbit $$\bO:=\bO_{(1^3;\emptyset)} \times \bO_{(1^4;\emptyset)} \times  \bO_{(\emptyset;1^4)} \times  \bO_{(\emptyset;1^2)}$$ for the Levi subgroup $L\cong \GL_{3} \times \GL_{4} \times \GL_{4} \times \GL_{2}$ up to $\GL_{13}$. We will illustrate this by picking representatives of $\bO$ and $\Ind_{\Vb_L}^\Vb (\bO)$, and showing they satisfy the requisite properties. \vspace{5pt}

We first construct a representative $(v,x_n) \in {\rm Ind}_{\Vb_L}^{\Vb}\bO$. From the bipartitions corresponding to $\bO$, we know that $ {\rm Ind}_{\Vb_L}^{\Vb}\bO$ has corresponding bipartition $\left((2^3,1);(2^2,1^2)\right)$. Therefore the element $x_n$ necessarily acts on $\k^{13}$ with Jordan blocks of size $\lambda= \mu+\nu=(4^2,3,2)$. Let 
$$
\{v_{ij} \; : \; 1 \leq i \leq 4, 1 \leq j \leq \lambda_{i}\}
$$
be a Jordan basis for $x_n$ on $\k^{13}$. We picture it via the following diagram, where $x_n$ sends the left-most box of every row to $0$, and every other box to its immediate left. 
\begin{eqnarray*}
\begin{tikzpicture}[scale=0.6]
\draw[-,line width=2pt] (0,4) to (0,0);

\draw (-2,4) -- (2,4) -- (2,3) -- (-2,3) -- cycle;
\draw (-2,3) -- (-2,2) -- (2,2) -- (2,3);
\draw (-2,2) -- (-2,1) -- (1,1); 
\draw (-1,4) -- (-1,0) -- (1,0) -- (1,4);

\node at (-1.5,3.5) {\tiny $v_{11}$};
\node at (-0.5,3.5) {\tiny $v_{12}$};
\node at (0.5,3.5) {\tiny $v_{13}$};
\node at (1.5,3.5) {\tiny $v_{14}$};

\node at (-1.5,2.5) {\tiny $v_{21}$};
\node at (-0.5,2.5) {\tiny $v_{22}$};
\node at (0.5,2.5) {\tiny $v_{23}$};
\node at (1.5,2.5) {\tiny $v_{24}$};

\node at (-1.5,1.5) {\tiny $v_{31}$};
\node at (-0.5,1.5) {\tiny $v_{32}$};
\node at (0.5,1.5) {\tiny $v_{33}$};

\node at (-0.5,0.5) {\tiny $v_{41}$};
\node at (0.5,0.5) {\tiny $v_{42}$};
\end{tikzpicture}
\end{eqnarray*}

If we define the elementary matrix operator $e_{(i,j),(k,l)}$ to act on the Jordan basis as:
$$
e_{\{(i,j),(k,l)\}}  v_{k,l} = \begin{cases}  v_{i,j} \; &\text{if} \; i=k,l=j+1, j \geq 1, \\  0  \; &\text{otherwise}, \end{cases} 
$$
then $x_n$ can be written as:
$$
x_n = \sum_{i=1}^{4} \sum_{j=1}^{\lambda_{i-1}}e_{\{(i,j),(i,j+1)\}}
$$
We may now choose $v$ to be the sum of the vectors to the left of the dividing wall in the above diagram. Thus $v= v_{11} + v_{12} + v_{21} + v_{22} + v_{31} + v_{32} + v_{41}$. It is now clear that the Jordan type of $x_n$ on the submodule $\k[x_n]v$ is $(2^2,1)=\mu$, and on the quotient $\k^{13}/\k[x_n] v$ is $(2^2,1^2)=\nu$. Thus the pair $(v,x_{n})$ is a representative for the induced orbit. \vspace{5pt} 

Finally, we express the representatives for the orbit $\bO$ in terms of the Jordan basis for the action of $x_n$ on $\k^{13}$. If we set $v^{(1)}= v_{11} +  v_{21} +  v_{31}$ and $v^{(2)}=  v_{12}  + v_{22} + v_{32} + v_{41}$, then 
$$
\left((v^{(1)},0), (v^{(2)},0), (0,0), (0,0)\right)
$$
is a canonical representative for $\bO$; where we have identified the columns as the standard bases for the natural modules in each component of the enhanced nilpotent cone for the Levi $L$. 
\end{example}

\subsection{The closure of a general enhanced Jordan class}

We are now in a position to describe the closure of arbitrary enhanced Jordan classes. We know from Corollary \ref{cor:all-induced}  that the  enhanced nilpotent orbit of $(v,x_n)$ in $\Vb$ is induced for  any $v\in \k^n$ and any $x_n\neq0$.
Thus, a similar statement holds for nilpotent orbits in the enhanced module $\Vb_L$ of  any Levi subgroup $L$ of $\G$.

\begin{proposition}\label{prop:closure-general}
Let $(v,x)=(v,x_s+x_n)\in\Vb$.  Let $L=\G^{x_s}$ with Lie algebra $\mf l$ and enhanced module $\Vb_L$.  Let $M$ be a  Levi subgroup of $L$ with enhanced module $\Vb_M$ and  $w\in\k^n$ be such that  $(v,x_n)\in{\rm Ind}_{\Vb_M}^{\Vb_L}(M(w,0))$. Let $Q$ be a parabolic subgroup of $\G$ adapted to $M$ and $v$,  let $\mf q={\rm Lie}(Q)$ and let ${\mf u}_Q$ be its nilradical.  Then  $\mathfrak z(\mathfrak l)\subseteq\mathfrak z(\mf m)$ and 
\begin{align}
\label{eq:closure-Jc-general}&\overline{\mf J(v,x)}= \bigcup_{z\in{\mf z}(\mf l)}\overline{{\rm Ind}_{\Vb_M}^{\Vb}(M(v,z))}= \G(\overline{M v},\mf z(\mf l)+{\mf u}_Q),\\
\label{eq:reg-closure-Jc-general}&{\mf J(v,x)}\subset \overline{\mf J(v,x)}^{\Gbreg}= \bigcup_{z\in{\mf z}(\mf l)}{\rm Ind}_{\Vb_M}^{\Vb}(M(v,z)).
\end{align}
\end{proposition}
\begin{proof}
By Lemma \ref{lem:induction0-step1} (3) we may  assume $v=w$.  

The closed subset  $(\overline{M v},\mf z(\mf l)+{\mf u}_Q)$ is irreducible and $Q$-stable and so the image $\G(\overline{M v},\mf z(\mf l)+{\mf u}_Q)$ of the action morphism $\G\times^Q(\overline{M v},\mf z(\mf l)+{\mf u}_Q)\to \Vb$ is closed and irreducible.  Invoking Lemma \ref{lem:induction0-step2}  we have
\begin{align*}
\dim \G(\overline{M v},\mf z(\mf l)+{\mf u}_Q)&\leq \dim \G\times^Q(\overline{M v},\mf z(\mf l)+{\mf u}_Q)\\
&=\dim \G-\dim Q+\dim Mv+\dim{\mf z}(\mf l)+\dim {\mf u}_Q\\
&=\dim\Vb-\codim_{\Vb_M}M(v,0)+\dim{\mf z}(\mf l)\\
&=\dim\Vb-\codim_{\Vb_L}{\rm Ind}_{\Vb_M}^{\Vb_L}M(v,0)+\dim{\mf z}(\mf l)\\
&=\dim\Vb-\codim_{\Vb_L}L(v,x_n)+\dim{\mf z}(\mf l)\\
&=\dim\G-\dim L^{(v,x_n)}+\dim{\mf z}(\mf l)\\
&=\dim\G-\dim {\G}^{(v,x)}+\dim{\mf z}(\mf l)\\
&=\dim{\mf J}(v,x).
\end{align*}
We  now show that $\mf J(v,x)\subseteq \G( \overline{M v},\mf z(\mf l)+{\mf u}_Q)$.  Indeed, by Proposition \ref{prop:centraliser}
\begin{align*}
\mf J(v,x)&=\G\left(v,{\mf z}(\mf l)^{\Gbreg}+x_n\right)= \bigcup_{z\in{\mf z}(\mf l)^{\Gbreg}}\G\left((0,z)+{\rm Ind}_{\Vb_M}^{\Vb_L}M(v,0)\right)\\
&\subseteq \bigcup_{z\in{\mf z}(\mf l)^{\Gbreg}}\G\left((0,z)+\overline{{\rm Ind}_{\Vb_M}^{\Vb_L}(M(v,0))}\right)\\
&\subseteq \bigcup_{z\in{\mf z}(\mf l)}\G\left((0,z)+\overline{{\rm Ind}_{\Vb_M}^{\Vb_{\G^z}}M(v,0)}\right)\\
&\subseteq \bigcup_{z\in{\mf z}(\mf l)}\overline{{\rm Ind}_{\Vb_M}^{\Vb}M(v,z)}= \bigcup_{z\in{\mf z}(\mf l)}\G(\overline{Mv},z+\mf u_Q)= \G(\overline{M v},\mf z(\mf l)+{\mf u}_Q)
\end{align*}
giving \eqref{eq:closure-Jc-general}. Then, \eqref{eq:reg-closure-Jc-general} follows from Lemma \ref{lem:induction0-step2} and Proposition \ref{prop:centraliser}.
\end{proof}

\begin{proposition}\label{prop:closure-enhanced}
The closure and the regular closure of enhanced Jordan classes are union of Jordan classes.
\end{proposition}
\begin{proof}
It is enough to show the statement for the closure of a Jordan class.  We retain notation from Proposition \ref{prop:closure-general}.  
Let  \begin{align*}
(w,y_s+y_n)\in  \overline{\mf J(v,x)}= \G(\overline{M v},\mf z(\mf l)+{\mf u}_Q).
\end{align*}
 Since Jordan classes are $\G$-stable, there is no loss of generality in assuming that $(w,y_s+y_n)\in ( \overline{M v},\mf z(\mf l)+{\mf u}_Q)$.  
 Then, $y_n\in {\mf u}_Q$ and $\gl_n^{y_s}\supset\mf l$, so  $\mf z(\gl_n^{y_s})^{\Gbreg}\subset \mf z(\gl_n^{y_s})\subset\mf z(\mf l)$. 
Hence, 
 \[\mf J(w,y_s+y_n)=\G\left(w,\mf z(\gl_n^{y_s})^{\Gbreg} +y_n\right)\subseteq  \G(\overline{M v},\mf z(\mf l)+{\mf u}_Q).\]
 \end{proof}

We are now in a position to extend Lemma \ref{lem:induction0-step2} and Proposition \ref{prop:unique} to arbitrary enhanced Jordan classes. 

\begin{corollary}\label{cor:induced-orbit}
Let $(v,x)=(v,x_s + x_n) \in \Vb$, let $L=\G^{x_s}$ and let ${\mf J}={\mf J}(v,x)$. 
Then, $\overline{{\mathfrak J}(v,x)}\cap{\mathcal N}(\Vb)$ is the closure of a $\G$-orbit $\bO$ in $\mathcal{N}(\Vb)$ satisfying 
$\codim_{\Vb}\bO=\codim_{\Vb_L}L (v,x)$. 
  
\end{corollary}
\begin{proof}Let  $M$  and ${\mf u}_Q$ be as in Proposition \ref{prop:closure-general}.
It follows from \eqref{eq:closure-Jc-general} and the Jordan decomposition that $\overline{{\mathfrak J}(v,x)}\cap{\mathcal N}(\Vb)=\G({\mf u}_Q, \overline{M v})$.
Since ${\mf u}_Q$ and $\overline{M v}$ are irreducible, $\overline{{\mathfrak J}(v,x)}\cap{\mathcal N}(\Vb)$ is irreducible and $\G$-stable. Hence, it is  the closure of an orbit $\bO$ in $\mathcal{N}(\Vb)$. In addition, \eqref{eq:reg-closure-Jc-general} guarantees that $\dim \bO=\dim \G (v,x)$, whence 
\begin{align}\label{eq:codim}\codim_{\Vb}\bO&=\dim \Vb-\dim\bO=n+\dim \G^{(v,x)}\\
\nonumber&=n+\dim L^{(v,x)}=\dim \Vb_L-\dim L^{(v,x)}=\codim_{\Vb_L}L (v,x).\end{align} 
\end{proof}

We call the nilpotent orbit $\bO$ from Corollary \ref{cor:induced-orbit} the \emph{induced enhanced nilpotent  orbit} of $L(v,x_n)$ and we denote it by ${\rm Ind}_{\Vb_L}^\Vb L (v,x_n)$.

\begin{proposition}
\label{prop:induction-transitive}
Induction of enhanced nilpotent orbits is transitive.
\end{proposition}
\begin{proof}Let $M$ and $L$ be Levi subgroups of $\G$ with $M\subset L$, let $\mf m$ and $\mf l$ be the respective Lie algebras, and let $\Vb_M$ and $\Vb_L$ be the respective enhanced modules.   Let $(v,x_n)\in \mathcal N(\Vb_M)$ and let $(w,y_n)\in{\rm Ind}_{\Vb_M}^{\Vb_L}M(v,x_n)$. Then
\begin{align*}
\overline{{\rm Ind}_{\Vb_L}^{\Vb}\left({\rm Ind}_{\Vb_M}^{\Vb_L}M(v,x_n)\right)}&=\mathcal N(\Vb)\cap\overline{\G(w,\mathfrak z(\mf l)^{\Gbreg}+y_n)}\\
&\subseteq \mathcal N(\Vb)\cap\overline{\G\left((0,\mathfrak z(\mf l))+\left(\mathcal N(\Vb_L)\cap \overline{L(v,\mf z(\mf m)^{L\operatorname{-reg}}+x_n)}\right)\right)}\\
&\subseteq \mathcal N(\Vb)\cap\overline{\G\left((0,\mathfrak z(\mf l))+L(v,\mf z(\mf m)+x_n)\right)}\\
&= \mathcal N(\Vb)\cap\overline{\G\left(L(v,\mf z(\mf l)+\mf z(\mf m)+x_n)\right)}\\
&= \mathcal N(\Vb)\cap\overline{\G\left(L(v,\mf z(\mf m)+x_n)\right)}\\
&=\mathcal N(\Vb)\cap\overline{\G\left(v,\mf z(\mf m)+x_n)\right)}=\overline{{\rm Ind}_{\Vb_M}^{\Vb}M(v,x_n)}.
\end{align*}
Iterated application of \eqref{eq:codim} shows that the first and the last terms in the sequence of inclusions have the same dimension, giving the statement.
\end{proof}

Induction is compatible with the embedding of $\Vb$ into $\Vb_{2n}$.

\begin{proposition}
\label{prop:double-induction}
Let $(\mathbf v,\mathbf x)=(\mathbf v,\mathbf x_s+\mathbf x_n)\in\Vb$ and $(v,x)=(v, x_s+x_n)=\psi(\varphi(\mathbf v,\mathbf x))\in \Vb_{2n}$.  Then 
\begin{equation*}\psi\left(\GL_{2n}\times^G\varphi\left(1\times^\G\left( {\rm Ind}_{\Vb_{\G^{\mathbf x_s}}}^{\Vb}\G^{x_s}(\mathbf v,\mathbf x_n)\right)\right)\right)={\rm Ind}_{\Vb_{\GL_{2n}^{x_s}}}^{\Vb}\GL_{2n}^{x_s}(v,x_n).\end{equation*} 
\end{proposition}
\begin{proof}
We denote by $\mathfrak J_{2n}(v,x)$ the enhanced Jordan class of $(v,x)$. Then
\begin{align*}
\psi\left(\GL_{2n}\times^G\varphi\left(1\times^\G\left(\overline{ \mathfrak J(\mathbf v,\mathbf x)}\cap\mathcal N(\Vb)\right)\right)\right)&
\subseteq\mathcal N(\Vb_{2n})\cap \psi\left(\GL_{2n}\times^G\varphi\left(1\times^\G\overline{\mathfrak{J}(\mathbf v,\mathbf x)}\right)\right)\\
&= \mathcal N(\Vb_{2n})\cap \overline{\psi(\GL_{2n}\times^G\varphi(1\times^\G\mathfrak J(\mathbf v,\mathbf x))}\\
&=\mathcal N(\Vb_{2n})\cap \overline{{\mathfrak J}_{2n}(v,x)} 
\end{align*}
and so $\psi\left(\GL_{2n}\times^G\varphi\left(1\times^\G\left( {\rm Ind}_{\Vb_{\G^{\mathbf x_s}}}^{\Vb}\G^{x_s}(\mathbf v,\mathbf x_n)\right)\right)\right)\subseteq\overline{{\rm Ind}_{\Vb_{\GL_{2n}^{x_s}}}^{\Vb}\GL_{2n}^{x_s}(v,x_n)}$.  Combining Proposition \ref{prop:double-dim} (2) and Corollary \ref{cor:induced-orbit} gives
\begin{align*}
\dim \GL_{2n}\times^G\left(1\times^\G\left( {\rm Ind}_{\Vb_{\G^{\mathbf x_s}}}^{\Vb}\G^{x_s}(\mathbf v,\mathbf x_n)\right)\right)
&=2\dim {\rm Ind}_{\Vb_{\G^{\mathbf x_s}}}^{\Vb}\G^{x_s}(\mathbf v,\mathbf x_n)\\
&=2\left(\dim \Vb-\dim\Vb_{\G^{\mathbf X_s}}+\dim \G^{\mathbf x_s}(\mathbf v,\mathbf x_n)\right)\\
&=\dim\GL_{2n}(0,x_s)+\dim \GL_{2n}^{x_s}(v,x)\\
&=\dim {\rm Ind}_{\Vb_{\GL_{2n}^{x_s}}}^{\Vb}\GL_{2n}^{x_s}(v,x_n)
\end{align*}
yielding the desired equality. 
\end{proof}

\begin{proposition}\label{prop:closure-double}
The assignment $\mathfrak J\mapsto \psi(\GL_{2n}\times^G\varphi(1\times^\G\mathfrak J))$ induces a injective poset morphism from the set of enhanced Jordan classes in $\Vb$ to the set of enhanced Jordan classes in $\Vb_{2n}$. 
\end{proposition}
\begin{proof}
Let $(\mathbf v,\mathbf x)=(\mathbf v,\mathbf x_s+\mathbf x_n)\in \Vb$ and  $L=\G^{\mathbf x_s}$. Let $M$ be a  Levi subgroup of $L$ and $\mathbf w\in\k^n$ be such that $(\mathbf v,\mathbf x_s)\in{\rm Ind}_{\Vb_M}^{\Vb_L}(M(\mathbf w,0))$ for some $\mathbf w\in\k^n$. 

Let respectively, $v, x_s,\,x_n, w,z\in \Vb_{2n}$ be the images of $\mathbf v, \, \mathbf x_s,\,\mathbf x_n, \, \mathbf w,\mathbf z$ through the natural inclusion of $\Vb$ into $\Vb_{2n}$ and, similarly, let $\widetilde{L},\,\widetilde{M}$ be the images of $L$ and $M$ in $\GL_{2n}$.  Propositions \ref{prop:closure-general} and \ref{prop:double-induction} yield
\begin{align*}
\psi\left(\GL_{2n}\times^G\left(\varphi\times^\G\overline{\mathfrak{J}(\mathbf v,\mathbf x)}\right)\right)&=\psi\left(\GL_{2n}\times^G\varphi\left(1\times^\G\left(\bigcup_{z\in{\mf z}(\mf l)}\overline{{\rm Ind}_{\Vb_M}^{\Vb}(M(\mathbf v,\mathbf z))}\right)\right)\right)\\
&=\bigcup_{z\in{\mf z}({\rm Lie}(\widetilde{L}))}\overline{{\rm Ind}_{\Vb_{\widetilde{M}}}^{\Vb_{2n}}(\widetilde{M}(v,z))}=\overline{\psi\left(\GL_{2n}\times^G\left(1\times^\G\mathfrak{J}(\mathbf v,\mathbf x)\right)\right)},
\end{align*}
giving the claim. 
\end{proof}

\subsection{The combinatorics of induction: the general case}
Let $(n_1, \ldots, n_r)$ be a composition of $n$. For each $i$, let $\lambda^{(i)}$ be a partition of $n_{i}$. Then $(\lambda^{(i)})_{i=1}^r$ is a sequence of partitions.  
We produce a partition $\lambda$ of $n$, denoted $\lambda^{\sum}$, where $\lambda_j = \sum_{i=1}^{r} \lambda_{j}^{(i)}$. This is simply adding together the partitions $\lambda^{(i)}$. From the definition of $\lambda^{\sum}$, it is clear that this operation is associative, in the sense that $\lambda^{\sum}$ is independent of the order by which individual parts of the partitions $\lambda^{(i)}$ are added together. This is a manifestation of transitivity of induction, which we will use to prove the following:

\begin{theorem}
	Let $(n_1, \ldots n_r)$ be a composition of $n$ that defines a Levi subgroup $L$ of $\GL_{n}$. Let $\big((\mu^{(i)}; \nu^{(i)})\big)_{i=1}^{r}$ be a sequence of bipartitions, where $(\mu^{(i)}; \nu^{(i)})$ is a bipartition of $n_i$. Let $\prod_{i=1}^r\bO_{(\mu^{(i)}; \nu^{(i)})}\subset \mathcal{N}(\Vb_L)$ be the corresponding nilpotent orbit for $L$. Then the bipartition corresponding to ${\rm Ind}_{\Vb_L}^{\Vb}(\prod_{i=1}^r\bO_{(\mu^{(i)}; \nu^{(i)})})$ is $(\mu^{\sum};\nu^{\sum})$. 
\end{theorem}

\begin{proof}
The $L$-orbit $\prod_{i=1}^r\bO_{(\mu^{(i)}; \nu^{(i)})}$ is a product of orbits where each factor $\bO_{(\mu^{(i)}; \nu^{(i)})}$ lies in the enhanced nilpotent cone $\k^{n_{i}} \times \mathcal{N}(\mathfrak{gl}_{n_{i}})$ for $\GL_{n_{i}}$, a factor in the Levi $L$. By Corollary~\ref{cor:all-induced}, the only rigid orbits in $\Vb$ correspond to partitions of the form $(1^n;\emptyset)$ and $(\emptyset, 1^n)$, and so it follows that we may regard each orbit $\bO_{(\mu^{(i)}; \nu^{(i)})}$ as induced from products of orbits whose corresponding bipartition is of that form. Now for each $i$, every partition $\mu^{(i)}$ and $\nu^{(i)}$ is the sum of partitions of the form $(1^{m(i,j)})$ and $(1^{q(i,l)})$ for $1 \leq j \leq \mu^{(i)}_{1}$ and $1 \leq l \leq \nu_{1}^{(i)}$, and where $m(i,j)$ is the $j$th part of the transpose partition $(\mu^{(i)})^{tr}$ and $q(i,l)$ is the $l$-th part of the transpose partition $(\nu^{(i)})^{tr}$ (this is saying nothing more than every partition can be thought of as the sum of columns in its corresponding Young diagram). Therefore each orbit $\bO_{(\mu^{(i)}; \nu^{(i)})}$ can be expressed as
$$
\bO_{(\mu^{(i)}; \nu^{(i)})} = \Ind_{\Vb_{L_{i}}}^{\Vb_{L}}\left( \prod_{j=1}^{\mu_{1}^{(i)}} \bO_{(1^{(m(i,j))}; \emptyset)} \times  \prod_{l=1}^{\nu_{1}^{(i)}} \bO_{(\emptyset; 1^{q(i,l)})}  \right)
$$
where $L_{i}$ is a Levi subgroup of $L$ in which the orbits in brackets are rigid. Therefore, by transitivity of induction we have
\begin{eqnarray*}
{\rm Ind}_{\Vb_L}^{\Vb}\prod_{i=1}^r\bO_{(\mu^{(i)}; \nu^{(i)})} &=& {\rm Ind}_{\Vb_L}^{\Vb}\ \left( \prod_{i=1}^{r}  \Ind_{\Vb_{L_{i}}}^{\Vb_{L}}\left( \prod_{j=1}^{\mu_{1}^{(i)}} \bO_{(1^{m(i,j)}; \emptyset)} \times  \prod_{l=1}^{\nu_{1}^{(i)}} \bO_{(\emptyset; 1^{q(i,l)})}  \right)\right) \\
&=& \prod_{i=1}^{r} {\rm Ind}_{\Vb_L}^{\Vb}  \Ind_{\Vb_{L_{i}}}^{\Vb_{L}}\left( \prod_{j=1}^{\mu_{1}^{(i)}} \bO_{(1^{m(i,j)}; \emptyset)} \times  \prod_{l=1}^{\nu_{1}^{(i)}} \bO_{(\emptyset; 1^{q(i,l)})}  \right) \\
&=& \prod_{i=1}^{r}   \Ind_{\Vb_{L_{i}}}^{\Vb}\left( \prod_{j=1}^{\mu_{1}^{(i)}} \bO_{(1^{m(i,j)}; \emptyset)} \times  \prod_{l=1}^{\nu_{1}^{(i)}} \bO_{(\emptyset; 1^{q(i,l)})}  \right).
\end{eqnarray*}
By Proposition~\ref{prop:combinatorics-inducedfromzero}, we know that the bipartition corresponding to this induced orbit is the component-wise sum of bipartitions $(1^{(m(i,j))}; \emptyset)$ and $(\emptyset; 1^{(q(i,l))})$, where $1 \leq i \leq r$, $1 \leq j \leq \mu^{(i)}_{1}$ and $1 \leq l \leq \nu_{1}^{(i)}$. Moreover, we also know that this process of summing is associative, and so we conclude that this bipartition is of the form $(\mu^{\sum};\nu^{\sum})$. 
\end{proof}

\section{Induction of orbits and closures of classes: exotic case}
\label{sec:inductionandclosures-exotic}

In this section we apply the correspondences from the previous ones and the results on enhanced Jordan classes to infer properties of exotic Jordan classes.

\begin{theorem}\label{thm:final-poset-bijection}
\begin{enumerate}
 \item The bijection established in Theorem \ref{thm:main-bijection} is an isomorphism of posets extending the poset isomorphism from Lemma \ref{lem:poset-nil}.
\item The closure and the regular locus of the closure of an exotic Jordan class is a union of exotic Jordan classes.
\item Let $J$ be an exotic class corresponding to a $G$-orbit of a pair $(M, \O)$ in accordance with Proposition~\ref{prop:basic-properties}. Then $\overline{J} \cap \cN(\V)$ is the closure of a single orbit, which we refer to as the exotic nilpotent $G$-orbit induced from $(M, \O)$, and we denote it by $\Ind_{\V_M}^{\V}(\O)$. For $\mathfrak J$ the enhanced Jordan class satisfying $J=\varphi(G\times^{\G}\mathfrak J)$, there holds 
\begin{equation*}\overline{\Ind_{\V_M}^{\V}(\O)}=\varphi(G\times^{\G}(\overline{\mathfrak J}\cap\mathcal N(\Vb)).\end{equation*} In particular, $\dim \overline{J}\cap\mathcal N(\V)=2\dim \overline{\mathfrak J}\cap\mathcal N(\Vb)$.
\end{enumerate}
\end{theorem}
\begin{proof}
(1) Let $\mathfrak J$ be an enhanced Jordan class in $\Vb$, and let $J:=\varphi(G\times^\G\mathfrak J)$.  We show that $\overline{J}=\varphi(G\times^\G\overline{\mathfrak J})$.
The restriction of $\varphi$ to  $1\times^{\G}\Vb $ is the natural inclusion  and $\varphi$ is $G$-equivariant, so  $\varphi(G\times^\G\overline{\mathfrak{J}})=G(\varphi(1\times^{\G}\overline{\mathfrak J}))\subseteq \overline{GJ}=\overline{J}$. 
We now apply $\psi(\GL_{2n}\times^G-)$ on both sides and invoke Proposition \ref{prop:closure-double}, obtaining 
\begin{align*}
\overline{\psi\left(\GL_{2n}\times^G\varphi\left(G\times^\G\mathfrak J\right)\right)}&=\psi\left(\GL_{2n}\times^G\varphi\left(G\times^\G\overline{\mathfrak J}\right)\right)\\
&\subseteq\psi\left(\GL_{2n}\times^G \overline{J}\right)\\
&=\GL_{2n} \psi\left(1\times^G\overline{J}\right)\\
&\subseteq \overline{\GL_{2n} \psi\left(1\times^G J\right)}\\
&= \overline{\psi\left(\GL_{2n}\times^G\varphi\left(1\times^\G \mathfrak J\right)\right)}
\end{align*}
where the last inclusion follows because the restriction of $\psi$ to $1\times^{\G}\V$ is the natural inclusion and $\psi$ is $\GL_{2n}$-equivariant.  Therefore all inclusions are equalities, whence
\begin{equation*}
\psi\left(\GL_{2n}\times^G\varphi\left(1\times^\G\overline{\mathfrak J}\right)\right)=\GL_{2n} \psi\left(1\times^G\overline{J}\right).
\end{equation*}
 Let $\bO$ be a $G$-orbit in $\overline{J}$. The $\GL_{2n}$-orbit $\psi(\GL_{2n}\times^G\bO)$ is contained in $\psi\left(\GL_{2n}\times^G\varphi\left(1\times^\G\overline{\mathfrak J}\right)\right)$. Its intersection with $\psi(1\times^G\V)$ is a $G$-orbit in $\varphi\left(1\times^\G\overline{\mathfrak J}\right)$ which must coincide with $\bO$ by Proposition \ref{prop:double-dim}.

(2)  Let $J$ be an exotic Jordan class. By Theorem \ref{thm:main-bijection} (1) and (2) there exists an enhanced Jordan class $\mathfrak{J}$ in $\Vb$ such that $J=\varphi(G\times^{\G}\mathfrak J)$ and $\overline{J}=\varphi(G\times^{\G}\overline{\mathfrak J})$. It follows from Proposition \ref{prop:closure-enhanced} that there are enhanced Jordan classes $\mathfrak{J}_1,\,\ldots,\,\mathfrak{J}_l$ and $\mathfrak{J}_{l+1},\,\ldots,\,\mathfrak{J}_{m}$ with $l\leq m$  in $\Vb$ such that
\begin{equation*}
\overline{\mathfrak J}=\bigcup_{i=1}^{m}{\mathfrak J}_i,\quad \overline{\mathfrak J}^{\Gbreg}=\bigcup_{i=1}^{l}{\mathfrak J}_i.
\end{equation*}
Therefore, using Proposition \ref{prop:double-dim}
\begin{align*}
\overline{J}&=\varphi\left(G\times^{\G}\bigcup_{i=1}^{m}{\mathfrak J}_i\right)=\bigcup_{i=1}^{m}\varphi\left(G\times^{\G}{\mathfrak J}_i\right),\\
\overline{J}^{G\operatorname{-reg}}&=\varphi\left(G\times^{\G}{\mathfrak J}^{\G\operatorname{-reg}}\right) = \varphi\left(G\times^{\G}\bigcup_{i=1}^{l}{\mathfrak J}_i\right)=\bigcup_{i=1}^{l}\varphi\left(G\times^{\G}{\mathfrak J}_i\right)
\end{align*}
where all $\varphi\left(G\times^{\G}{\mathfrak J}_i\right)$ are exotic Jordan classes by Theorem \ref{thm:main-bijection} (1). 

(3) Let $d\in \mathbb N$ be such that $J\subset \V_{(d)}$. 
On the one hand we have 
\begin{align*}
\varphi(G\times^{\G}(\overline{\mathfrak J}\cap\mathcal N(\Vb)))&\subseteq \varphi\left((G\times^{\G}\overline{\mathfrak J})\cap (G\times^{\G}\mathcal N(\Vb))\right)\\
&\subseteq \varphi \left(G\times^{\G}\overline{\mathfrak J}\right)\cap \varphi(G\times^{\G}\mathcal N(\Vb))\subseteq \overline{J}\cap\mathcal N(\V).
\end{align*}
By Corollary \ref{cor:induced-orbit} the intersection $\overline{\mathfrak J}\cap\mathcal N(\Vb)$ is the closure of a single nilpotent $\G$-orbit $\bO$ of dimension $(1/2)d$. 
Hence, the nilpotent  $G$-orbit $\O':=\varphi(G\times^\G {\bO})$  satisfies $\overline{\O'}\subseteq \left(\overline{J}\cap\mathcal N(\V)\right)$. 

Let now $\widetilde{\mathfrak J}:=\psi(\GL_{2n}\times^G J)$.
Arguing as above we see that 
\begin{equation}\label{eq:chain}
\psi(\GL_{2n}\times^G\O')\subseteq \psi(\GL_{2n}\times^G( \overline{J}\cap\mathcal N(\V)) )\subseteq \overline{\widetilde{\mathfrak J}}\cap \mathcal N(\Vb_{2n}).
\end{equation}
 Invoking Corollary \ref{cor:induced-orbit} once more, we see that $ \overline{\widetilde{\mathfrak J}}\cap \mathcal N(\Vb_{2n})$ is the closure of a single nilpotent $\GL_{2n}$-orbit $\widetilde{\bO}$ with  
$\dim \widetilde{\bO}= 2d$.
 Since $\dim \psi(\GL_{2n}\times^G\O')=2d$, we have $\psi(\GL_{2n}\times^G\overline{\O'})=\overline{\widetilde{\bO}}$. 
Passing to the regular locus in \eqref{eq:chain} we obtain 
\begin{equation*}\psi(\GL_{2n}\times^G\O')= \psi(\GL_{2n}\times^G( (\overline{J}\cap\mathcal N(\V))^{G\operatorname{-reg}}))= {\widetilde{\bO}}.\end{equation*}
By Lemma \ref{lem:poset-nil} any exotic nilpotent orbit in $\overline{J}\cap\mathcal N(\V)$ is contained in the closure of $\O'$.
\end{proof}

\begin{rem}
\label{ss:properties-exotic}
{\rm Thanks to Theorem~\ref{thm:final-poset-bijection} several key properties of induced exotic orbits and exotic Jordan classes now follow from the enhanced setting. In particular:
\begin{itemize}
\item The closure of an exotic Jordan class can be described precisely using induction (Proposition~\ref{prop:closure-general}).
\item The closure and regular closure of an exotic Jordan  class are each a union of exotic Jordan classes. In particular, $\V$ is stratified by its exotic Jordan classes (Proposition~\ref{prop:closure-enhanced}).
\item Induction of exotic nilpotent orbits is transitive and preserves codimension (Corollary~\ref{cor:induced-orbit}, Proposition~\ref{prop:induction-transitive})
\end{itemize}}
\end{rem}

\begin{rem}
\label{rem:classification-sheets}
{\rm It follows from \eqref{eq:reg-closure-Jc-general} and Corollary~\ref{cor:induced-orbit} that if $\J_i = \J(L_i, \bO_i) \subseteq \Vb_{(d)}$, for $i = 1,2$ and $d\in \mathbb N$ are two enhanced classes  then $\J_1 \subseteq \overline{\J_2}$ if and only if, upto conjugacy, $L_2 \subseteq L_1$ and $\Ind_{\Vb_{L_2}}^{\Vb_{L_1}}(\bO_2) = \bO_1$. It follows immediately that a Jordan class $\J = \J(L, \bO)$ is dense in a sheet if and only if $\bO \subseteq \Vb_L$ is a rigid orbit.

By Proposition~\ref{prop:basic-properties}(3), Jordan classes are classified by  a partition $\lambda$ and a choice of bipartition $(\mu_i; \nu_i)$ of $\lambda_i$ for each $i$ (upto reordering by $\mathbb{S}_\lambda$). Combining with Corollary~\ref{cor:all-induced}, we see that when $\J$ is dense in a sheet, the bipartitions are each of the form $((1^{\lambda_i}); \emptyset)$ or $(\emptyset ; (1^{\lambda_i}))$.

Thanks to Theorem~\ref{thm:final-poset-bijection} anaologous remarks hold for exotic Jordan classes. This concludes the proof of Theorem~D of the introduction.}
\end{rem}

\section*{Acknowledgements}
G.C. is grateful to Paul Levy for a  useful discussion concerning symplectic singularities and nilpotent orbits that lead to Remark \ref{rem_exoticnotccs}.

Project partially funded by the European Union -- Next Generation EU under the National
Recovery and Resilience Plan (NRRP), Mission 4 Component 2 Investment 1.1 -
Call PRIN 2022 No. 104 of February 2, 2022 of Italian Ministry of
University and Research; Project 2022S8SSW2 (subject area: PE - Physical
Sciences and Engineering) ``Algebraic and geometric aspects of Lie theory''.
F.A. and G.C.  and F.E. and are members of the INdAM group GNSAGA.
F.A.'s research was conducted at FSU Jena. L. T. is funded by a UKRI Future Leaders Fellowship MR/Z000394/1, and during the initial phase of this research he was supported by grant numbers MR/S032657/1,
MR/S032657/2, MR/S032657/3.

\setcounter{tocdepth}{1}
\bibliographystyle{plain}

\begin{thebibliography}{10}
\bibitem{ah} P. N. Achar and A. Henderson, Orbit closures in the enhanced nilpotent cone, {\it Adv. Math.} {\bf 219} (2008), no. 1, 27–62. 
\bibitem{ahj} P.~Achar, A.~Henderson, B.~Jones,
Normality of orbit closures in the enhanced nilpotent cone,
{\it Nagoya Math. J.} {\bf 203} (2011), 1--45.
\bibitem{Antor_g2} J.~Antor, Geometric realizations of affine Hecke algebras with unequal parameters,     (2025), https://arxiv.org/abs/2505.02942.
\bibitem{Antor_f4} J.~Antor, An exotic Springer correspondence for $F_4$, (2025), https://arxiv.org/abs/2508.14199.
\bibitem{BR} P. Bardsley, R.W. Richardson, \'Etale slices for algebraic transformation groups in characteristic $p$,  {\it Proc. Lond. Math. Soc.} (3) {\bf 51} (1985), no. 2, 295--317.
\bibitem{Beau}A. Beauville, Symplectic singularities. Invent. math. {\bf 139} (2000), 541--549. 
\bibitem{BB}G. Bellamy, M. Boos, The (cyclic) enhanced nilpotent cone via quiver representations, {\it manuscripta math.} {\bf 161}, 333--362 (2020). 
\bibitem{bernstein} J. Bernstein, $P$-invariant distributions on $\mathrm{GL}(N)$ and the classification of unitary representations of $\mathrm{GL}(N)$ (non-archimedean case), In: Lie Group representations II, Proceedings of the Special Year held at the University of Maryland, College Park 1982--1983, LNM 1041, R. Herb, S. Kudla, R. Lipsman, J. Rosenberg (Eds.) 50--102, Springer-Verlag 1984. 
\bibitem{borel}A. Borel, Linear Algebraic Groups, GTM vol. 126,  Springer-Verlag. (1991)  
\bibitem{bo} W. Borho, Zum Induzieren unipotenter Klassen, {\it Math. Sem. Univ. Hamburg} {\bf 50}, 1--4 (1981).
\bibitem{bo-inv}W. Borho,  \"{U}ber Schichten halbeinfacher Lie-Algebren, {\it Invent. math.} {\bf 65}, 283--317 (1981)
\bibitem{BorhoKraft} W. Borho and H. Kraft,  \"Uber Bahnen und deren Deformationen bei linearen Aktionen reduktiver Gruppen,  \"{U}ber Bahnen und deren Deformationen bei linearen Aktionen  reduktiver Gruppen, {\it Comment. Math. Helv.}, (1979) {\bf 54}, 61--104.
\bibitem{broer} A. Broer, Decomposition varieties in semisimple Lie algebras, {\it Can. J. Math.} {\bf 50} (5), 1998, pp. 929--971.
\bibitem{broer2} A. Broer, Lectures on decomposition classes,  In: Broer, A., Daigneault, A., Sabidussi, G. (eds) Representation Theories and Algebraic Geometry. Nato ASI Series, vol 514. Springer, Dordrecht.
\bibitem{Bul} M. Bulois, Sheets of symmetric Lie algebras and Slodowy slices,
{\it J. Lie Theory} {\bf 21} (2011), no. 1, 1--54.
\bibitem{BH} M. Bulois, P. Hivert, Sheets in symmetric Lie algebras and slice induction, {\it Transform. Groups} {\bf 21} (2016), no. 2, 355--375.
\bibitem{CG} N. Chriss, V. Ginzburg, Representation theory and complex geometry,
Mod. Birkh\"auser Class. Birkh\"auser Boston, Ltd., Boston, MA, 2010, x+495 pp.
\bibitem{DK} J. Dadok and V. G. Kac, Polar representations, {\it J. Algebra},  {\bf 92} (2), 1985, 504--524.
\bibitem{GG} W. L. Gan and V. Ginzburg, Almost-Commuting Variety, $\mathcal{D}$-Modules, and Cherednik Algebras, {\it IMRP Int. Math. Res. Pap.} Volume {\bf 2006}, Article ID 26439, Pages 1--54
\bibitem{GV} V. Gatti, E. Viniberghi, Spinors of $13$-dimensional space, {\it Adv. Math.} {\bf 30} (2), 1978, 137--155.
\bibitem{hartshorne} R. Hartshorne, Algebraic Geometry, GTM 32, Springer-Verlag, 1977.
\bibitem{JaNO}{J.C.~Jantzen}, Nilpotent orbits in representation theory, in: B.\,Orsted (ed.), ``Representation and Lie theory'', Progr. in
Math., {\bf 228}, 1--211, Birkh\"auser, Boston 2004.
\bibitem{Kac}V. G. Kac. Infinite root systems, representations of graphs and invariant theory. II. {\it J. Algebra},
{\bf 78} (1), 141--162, 1982.
\bibitem{K1} S. Kato,  An exotic Deligne-Langlands correspondence for symplectic groups. {\it Duke Math. J.} {\bf 148} (2), 305--371 (2009).
\bibitem{K2} \bysame,
Deformations of nilpotent cones and Springer correspondences. {\it Amer. J. Math.} {\bf 133} (2011), no. 2, 519--553.
\bibitem{KR71}{B. Kostant, S. Rallis}, Orbits and representations associated with symmetric spaces, {\it Amer. J. Math.} {\bf 93} (1971), 753--809.
\bibitem{Le09}{P. Levy}, Vinberg's  $\theta$-groups in positive characteristic and Kostant-Weierstrass slices, {\it Transform. Groups} {\bf 14} (2009), no. 2, 417--461.
\bibitem{LS} G. Lusztig, N. Spaltenstein, Induced nilpotent classes, {\it J. London Math. Soc.} {\bf 19} (1979), 41?52.
\bibitem{Ma} C. Mautner, Affine pavings and the enhanced nilpotent cone. {\it Proc. Amer. Math. Soc.} {\bf 145}, no. 4 (2017): 139398. 
\bibitem{milne} J.S. Milne,  Lectures on \'etale cohomology,  v 2.21, 2013, https://www.jmilne.org/math/CourseNotes/LEC.pdf
\bibitem{Na}{Y. Namikawa},On the structure of homogeneous symplectic varieties of complete intersection, Invent. Math. {\bf 193} (2013), no.~1, 159--185.
\bibitem{Oh86}{T. Ohta}, The singularities of the closures of nilpotent orbits in certain symmetric pairs, Tohoku
Mathematical Journal, Second Series {\bf 38} (3), 441--468 (1986).
\bibitem{PV94}{V. L. Popov, E. B. Vinberg}, Algebraic geometry. IV
Encyclopaedia Math. Sci., 55
Springer-Verlag, Berlin, 1994, vi+284 pp.
\bibitem{PS}{A. Premet, D. Stewart}, Rigid orbits and sheets in reductive Lie algebras over fields of prime characteristic,
{\it J. Inst. Math. Jussieu} {\bf 17} (2018), no. 3, 583--613.
\bibitem{Spalt82}{N.~Spaltenstein}, Nilpotent classes and sheets of Lie algebras in bad characteristic, Math. Z. {\bf 181} (1982), no.~1, 31--48;
\bibitem{Spalt84}{N.~Spaltenstein}, Nilpotent classes in Lie algebras of type $F\sb{4}$\ over fields of characteristic $2$, J. Fac. Sci. Univ. Tokyo Sect. IA Math. {\bf 30} (1984), no.~3, 517--524;
\bibitem{springer}{T.A. Springer}, {The exotic nilcone of a symplectic group,}  J. Algebra {\bf 321} (11) (2009), 3550--3562. 
\bibitem{SF80} H. Strade, R. Farnsteiner,
Modular Lie algebras and their representations,
Monogr. Textbooks Pure Appl. Math., 116 Marcel Dekker, Inc., New York, 1988, x+301 pp.
\bibitem{sun}M. Sun, Point stabilisers for the enhanced and exotic nilpotent cones. {\it J. Group Theory}
14(6), 825--839 (2011).
\bibitem{mirabolic} R. Travkin, Mirabolic Robinson-Schensted-Knuth correspondence, {\it Selecta Math. (N.S.)} {\bf 14} (2009), no. 3-4, 727-758.
\end{thebibliography}

\end{document}